\documentclass[12pt]{amsart}%
\usepackage{amsfonts}
\usepackage{amsmath}
\usepackage{amssymb}
\usepackage{amscd}
\usepackage{graphicx}

\newtheorem{theorem}{Theorem}
\newtheorem{lemma}{Lemma}
\newtheorem{proposition}{Proposition}
\newtheorem{remark}{Remark}

\newtheorem{definition}{Definition}
\newtheorem{corollary}{Corollary}
\newtheorem{example}{Example}

\numberwithin{equation}{section}
\begin{document}
\title[Convergence of the Sasaki-Ricci Flow on Sasakian $5$-Manifolds]{Convergence of the Sasaki-Ricci Flow on Sasakian $5$-Manifolds of General Type}
\author{$^{\ast}$Shu-Cheng Chang}
\address{Department of Mathematics, National Taiwan University, Taipei, Taiwan}
\email{scchang@math.ntu.edu.tw}
\author{$^{\dag}$Yingbo Han}
\address{{School of Mathematics and Statistics, Xinyang Normal University}\\
Xinyang,464000, Henan, P.R. China}
\email{{yingbohan@163.com}}
\author{$^{\ast\ast}$Chien Lin}
\address{Mathematical Science Research Center, Chongqing University of Technology,
400054, Chongqing, P.R. China}
\email{chienlin@cqut.edu.cn}
\author{$^{\ast\ast\ast}$Chin-Tung Wu}
\address{Department of Applied Mathematics, National Pingtung University, Pingtung
90003, Taiwan}
\email{ctwu@mail.nptu.edu.tw }
\thanks{$^{\ast}$Shu-Cheng Chang and $^{\ast\ast\ast}$Chin-Tung Wu are partially
supported in part by the MOST of Taiwan. $^{\dag}$Yingbo Han is partially
supported by an NSFC 11971415 and Nanhu Scholars Program for Young Scholars of
{Xinyang Normal University}.}

\begin{abstract}
In this paper, we show that the uniform $L^{4}$-bound of the transverse Ricci
curvature along the Sasaki-Ricci flow on a compact quasi-regular Sasakian
$(2n+1)$-manifold $M$ of general type. As an application, any solution of the
normalized Sasaki-Ricci flow converges in the Cheeger-Gromov sense to the
unique singular Sasaki $\eta$-Einstein metric on the transverse canonical
model $M_{\mathrm{can}\text{ }}$ of $M$ \ if $n\leq3$. In particular for
$n=2,$ $M_{\mathrm{can}\text{ }}$ is a $S^{1}$-orbibundle over the unique
K\"{a}hler-Einstein orbifold surface $(Z_{\mathrm{can}},\omega_{KE})$ with
finite point orbifold singularities. The floating foliation $(-2)$-curves in
$M$ will be contracted to orbifold points by the Sasaki-Ricci flow as
$t\rightarrow\infty.$

\end{abstract}
\maketitle
\tableofcontents

\section{Introduction}

Let $(M,\eta,\xi,\Phi,g)$ be a compact quasi-regular Sasakian $(2n+1)$%
-manifold. Then by the first structure theorem (Proposition \ref{P21}),
$M$\ is a principal $S^{1}$-orbibundle ($V$-bundle) over $Z$ which is also a
$Q$-factorial, polarized, normal projective orbifold variety such that there
is an orbifold Riemannian submersion$\ $%
\[
\pi:(M,g,\omega)\rightarrow(Z,h,\omega_{h})
\]
with $\omega=\pi^{\ast}(\omega_{h}).$

If the orbifold structure $(Z,\Delta)$ of the leave space $Z$ is well-formed,
then its orbifold singular locus and algebro-geometric singular locus
coincide, equivalently $Z$ has no branch divisors with $\Delta=\emptyset$.

In the case of $n=2,$ $Z$ has isolated singularities of a finite cyclic
quotient of type $\frac{1}{r}(1,a).$ In particular, it is Kawamata log
terminal singularities. The corresponding singularities in $(M,\eta,\xi
,\Phi,g)$\ is the\ foliation cyclic quotient singularities of type $\frac
{1}{r}(1,a)$ at a singular fibre $S^{1}$ in $M$ (Theorem \ref{T21}). The
orbifold canonical divisor $K_{Z}^{orb}$ and canonical divisor $K_{Z}$ are the
same and then
\[
K_{M}^{T}=\pi^{\ast}(\varphi^{\ast}K_{Z}).
\]

Note that the class of simply connected, closed, oriented, smooth,
$5$-manifolds is classifiable under diffeomorphism due to Smale-Barden
(\cite{s}, \cite{b}). Then, in this paper, it is our goal to focus on a
classification of compact quasi-regular Sasakian $5$-manifolds according to
the global properties of the Reeb\textbf{\ }$U(1)$-fibration\emph{. }

More precisely, there is the Sasaki analogue of Mori's minimal model program
with respect to $K_{Z}$ in a such compact quasi-regular Sasakian $5$-manifold.
In other word, find a finite sequence of basic transverse birational maps
$f_{1},\cdots,f_{k\text{ }}$ and Sasakian $5$-manifolds $M_{1},\cdots,M_{k}$
with
\[%
\begin{array}
[c]{c}%
M=M_{0}\overset{f_{1}}{\rightarrow}M_{1}\overset{f_{2}}{\rightarrow}%
M_{2}\overset{f_{3}}{\rightarrow}\cdots\overset{f_{k}}{\rightarrow}M_{k}%
\end{array}
\]
so that either $M_{k}$ is transverse minimal model or Mori fibre space. That
is, to find $f_{i}$ which "remove" $K_{M}^{T}$-negative foliation curves $V$
with $K_{M_{i}}^{T}\cdot V<0.$ In the paper of \cite{clw}, we proved that
there exists a finite sequence of foliation extremal ray contractions
\[
f_{i}:M_{i-1}\rightarrow M_{i}\text{,\ }i=1,\cdots,k
\]
such that every $M_{i}$ is a Sasakian manifold having at worst foliation
cyclic quotient singularities and for every $f_{i}$ one of the following holds:

(A) Foliation divisorial contraction (The locus of the foliation extremal ray
is an irreducible basic divisor): $f_{i}$ is a foliation divisorial
contraction of a foliation curve $V$ with $V^{2}<0$, and the Picard number
satisfies $\rho(Z_{i})=\rho(Z_{i-1})-1;$ or

(B) Foliation fibre contraction (transverse Mori fibre space) (The locus of
the foliation extremal ray is $M_{i-1}$): $f_{i}$ is a singular fibration such
that either

(i) there is a map%

\[
f:M_{k}\rightarrow pt,
\]
then $K_{M_{k}}^{T}<0$ and thus $M_{k}$ is transverse minimal Fano and the
leave space $Z_{k}$ is minimal log del Pezzo surface of $\frac{1}{r}%
(1,a)$-type singularities and Picard number one, or

(ii)
\[
f:M_{k}\rightarrow\Sigma_{h},
\]
then $M_{k}$ is an $S^{1}$-orbibundle of a rule surface over Riemann surfaces
$\Sigma_{h}$\ of genus $h$, or

(C) $M_{k}$ is nef: $f_{i}=f_{k}$, and $M_{k}$ has at worst foliation cyclic
quotient singularities and has no foliation $K_{M}^{T}$-negative curves.

The Mori's minimal model program in birational geometry can be viewed as the
complex analogue of Thurston's geometrization conjecture which was proved via
Hamilton's Ricci flow with surgeries on $3$-dimensional Riemannian manifolds
by Perelman \cite{p1}, \cite{p2}, \cite{p3}. Likewise, there is a conjecture
picture by Song-Tian \cite{st} that the K\"{a}hler-Ricci flow should carry out
an analytic minimal model program with scaling on projective varieties.
Recently, Song and Weinkove \cite{sw1} established the above conjecture on a
projective algebraic surface.

On the other hand, Sasakian manifolds can be view as an odd-dimensional
analogous of K\"{a}hler manifolds and the Sasaki-Ricci flow can be viewed as a
Sasaki analogue of Cao's result (\cite{cao}) for the K\"{a}hler-Ricci flow. In
the paper of Smoczyk-Wang-Zhang \cite{swz}, they introduced such a flow and
proved that the flow has the longtime solution and asymptotic converges to a
Sasaki $\eta$-Einstein metric when the basic first Chern class is null
($c_{1}^{B}(M)=0$) or negative ($c_{1}^{B}(M)<0$). The latter case is
equivalent to the condition of the transverse canonical line bundle $K_{M}%
^{T}$ is ample.

In this paper, we consider the case where $K_{M}^{T}$ is not necessarily
ample, but nef and big. Such a Sasakian manifold is known as a smooth
transverse minimal model of general type. It is served as an odd-dimensional
counterpart of the K\"{a}hler-Ricci flow on K\"{a}hler surfaces of general
type as in \cite{tz2} and \cite{gsw} via the Sasaki-Ricci flow on a compact
quasi-regular Sasakian $5$-manifold.

More precisely, by applying the Sasaki analogue of arguments in \cite{tz2}, we
show that the $L^{4}$-norm of the transverse Ricci curvature is uniformly
bounded along the normalized Sasaki-Ricci flow on any transverse minimal model
of general type and derive the following results in the paper.

\begin{theorem}
\label{T11} Let $(M,\eta_{0},\xi_{0},\Phi_{0},g_{0},\omega_{0})$ be a compact
quasi-regular Sasakian $(2n+1)$-manifold and its leave space $Z$ of the
characteristic foliation be well-formed. Suppose that $(M,\eta_{0},\xi
_{0},\Phi_{0},g_{0},\omega_{0})$ is a smooth transverse minimal model of
general type with dimension $n\leq3$ and $\omega(t)$ is a solution to the
normalized Sasaki-Ricci flow%
\begin{equation}%
\begin{array}
[c]{c}%
\frac{\partial}{\partial t}\omega(t)=-\mathrm{Ric}_{\omega(t)}^{T}%
-\omega(t),\text{ }\omega(0)=\omega_{0}.
\end{array}
\label{0}%
\end{equation}
Then $(M,\omega(t))$ converges in the Cheeger-Gromov sense to the unique
singular $\eta$-Einstein metric $\omega_{\infty}$ on the transverse canonical
model $M_{\mathrm{can}}$ of $M$ which is a $S^{1}$-orbibundle over the unique
singular K\"{a}hler-Einstein normal projective variety $(Z_{\mathrm{can}%
},\omega_{KE})$. Here $Z_{\mathrm{can}}$ is the canonical model of $Z.$
\end{theorem}

\begin{remark}
1. Note that the same as in the K\"{a}hler-Ricci flow, we have the same
transversal holomorphic foliation ($\xi$ is fixed) but with the new transverse
K\"{a}hler structure under the Sasaki-Ricci flow.

2. By the second
structure theorem on a Sasakian manifold, %
if $M$ admits an irregular Sasakian structure, it admits many locally free
circle actions which is our starting point for a quasi-regular case.

3. If $c_{1}^{B}(M)\leq0,$ then $K_{M}^{T}$ is nef and semi-ample. Moreover,
$(M,\eta_{0},\xi_{0},\Phi_{0},g_{0},\omega_{0})$ will be a compact
quasi-regular Sasakian $5$-manifold (\cite[Theorem 8.1.14]{bg}). Thus
$K_{M}^{T}$ is nef and big can be replaced by $c_{1}^{B}(M)\leq0$ and
$(c_{1}^{B}(M))^{2}>0$ in a compact Sasakian $5$-manifold without the
assumption of quasi-regularity.

4. In this paper, we assume that $M$\ is a compact quasi-regular transverse
Sasakian manifold and the space $Z$ of leaves is well-formed which means its
orbifold singular locus and algebro-geometric singular locus coincide,
equivalently $Z$ has no branch divisors. However when its leave space $Z$ is
not well-formed, we conjecture that there is a Sasaki analogue of analytic Log
minimal model program with respect to $K_{Z}+[\Delta]$ via the conical
Sasaki-Ricci flow. This is the odd dimensional counterpart of the conical
K\"{a}hler-Ricci flow (\cite{lz}, \cite{shen}). We hope to address this issue
in the near future.
\end{remark}

For $n=2,$ one can show that the limit is a smooth K\"{a}hler-Einstein
orbifold $Z_{\mathrm{can}}$ with finite orbifold points by a classical
argument of removing isolated singularities due to Anderson \cite{a},
Bando-Kasue-Nakajima \cite{bkn}, and Tian \cite{t}. More precisely, if $M$ is
a compact quasi-regular Sasakian $5$-manifold ($n=2$), then the singular set
$S\subset Z_{\infty}$ is dimension $0$ and $h_{\infty}$ will be an orbifold
K\"{a}hler-Einstein metric on $Z_{\infty}.$ More precisely, the solution
$\omega(t)$ of the normalized Sasaki-Ricci flow on $M$ starting with any
initial Sasakian metric $\omega_{0}$ is continuous through finitely many
contraction surgeries (\cite{clw}) in the Gromov-Hausdorff topology for
$t\in\lbrack0,\infty)$ and converges in the Cheeger-Gromov sense to the unique
Sasaki $\eta$-Einstein orbifold metric on the canonical model $M_{\mathrm{can}%
\text{ }}$ of $M$ which is a $S^{1}$-orbibundle over the unique
K\"{a}hler-Einstein orbifold surface $(Z_{\mathrm{can}\text{ }},\omega_{KE})$
with finite point orbifold singularities.

\begin{corollary}
\label{C11}Let $(M,\eta_{0},\xi_{0},\Phi_{0},g_{0},\omega_{0})$ be a compact
quasi-regular Sasakian $5$-manifold and its leave space $Z$ of the
characteristic foliation be well-formed. Suppose that $(M,\eta_{0},\xi
_{0},\Phi_{0},g_{0},\omega_{0})$ is a smooth transverse minimal model of
general type and $\omega(t)$ is a solution to the normalized Sasaki-Ricci flow
(\ref{0}). Then $(M,\omega(t))$ converges in the Cheeger-Gromov sense to the
unique Sasaki $\eta$-Einstein orbifold metric $\omega_{\infty}$
\[%
\begin{array}
[c]{c}%
\mathrm{Ric}_{\omega_{\infty}}^{T}=-\omega_{\infty}%
\end{array}
\]
on the transverse canonical model $M_{\mathrm{can}\text{ }}$ with finite
orbifold foliation singularities at a singular fibre $S^{1}$ on $M$. In
particular, the floating foliation $(-2)$-curves in $M$ will be contracted to
orbifold points by the Sasaki-Ricci flow as $t\rightarrow\infty.$
\end{corollary}

Note that in our proof of the $L^{4}$-norm bound of the transverse Ricci
curvature\ on a compact quasi-regular Sasakian $(2n+1)$-manifold, all the
integrands are only involved with the transverse K\"{a}hler structure
$\omega(t)$ and basic sections. Then one expects that, under the Sasaki--Ricci
flow, the expressions involved behave essentially the same as in the
K\"{a}hler-Ricci flow when one applies the Weitzenb\H{o}ch type formulae and
integration by parts.

More precisely, our proof relies on the Cheeger-Colding-Tian \cite{cct}
regularity theorem for K\"{a}hler orbifolds and the uniform $L^{4}$-bound of
transverse Ricci curvature under the Sasaki-Ricci flow (\ref{0}) on a compact
quasi-regular Sasakian manifold where the transverse canonical line bundle is
nef and big. In the last section, we will add some remark about the Sasaki
analogue of Guo-Song-Weinkove \cite{gsw} arguments for the contraction on the
floating foliation $(-2)$-curves.

In section $3$ and section $4,$ we give some fundamental estimates for the
Sasaki-Ricci flow. In section $5,$ we derive the estimate on the $L^{4}$-bound
of the transverse Ricci curvature under the normalized Sasaki-Ricci flow. In
the last section, we give a proof of our main theorem by applying the
Cheeger-Colding-Tian structure theory for K\"{a}hler orbifolds (\cite{cct},
\cite{tz1} and \cite[Theorem 2.3]{tz2}) to study the structure of desired
limit spaces. At the end, we will add some remark about the Sasaki analogue of
Guo-Song-Weinkove \cite{gsw} arguments for the contraction on foliation $(-2)$-curves.

For a completeness, we give some preliminaries on structures theorems for
Sasakian structures, foliation normal local coordinates, basic transverse
holomorphic line bundles and its associated basic divisors, and the type of
singularities in Sasakian manifolds in the section $2$ and Appendix.

\section{Preliminaries}

We will address the preliminary notions on the foliation normal coordinate and
basic cohomology and Type II deformation in a Sasakian manifold. We refer to
\cite{bg}, \cite{fow}, and references therein for some details. We will also
address on the Sasakian structure, the leave space and its foliation
singularities, basic holomorphic line bundles and basic divisors over Sasakian
manifolds in the appendix for the completeness.

\subsection{Sasakian Manifolds}

Let $(M,g,\nabla)$ be a Riemannian $(2n+1)$-manifold. We say $(M,g)$ is called
Sasaki if the cone
\[
(C(M),\overline{g}):=(\mathbb{R}^{+}\times M\mathbf{,}dr^{2}+r^{2}g)
\]
such that $(C(M),\overline{g},J,\overline{\omega})$ is K\"{a}hler with
\[%
\begin{array}
[c]{c}%
\overline{\omega}=\frac{1}{2}\sqrt{-1}\partial\overline{\partial}r^{2}.
\end{array}
\]
The function $\frac{1}{2}r^{2}$ is hence a global K\"{a}hler potential for the
cone metric. As $\{r=1\}=\{1\}\times M\subset C(M)$, we may define%
\[%
\begin{array}
[c]{c}%
\overline{\xi}=J(r\frac{\partial}{\partial r})
\end{array}
\]
and the Reeb vector field $\xi$ on $M$
\[%
\begin{array}
[c]{c}%
\xi=J(\frac{\partial}{\partial r}).
\end{array}
\]
Also
\[%
\begin{array}
[c]{c}%
\overline{\eta}(\cdot)=\frac{1}{2}\overline{g}(\xi,\cdot)
\end{array}
\]
and the contact $1$-form $\eta$ on $TM$
\[
\eta(\cdot)=g(\xi,\cdot).
\]
Then $\xi$ is the killing vector field with unit length such that
\[
\eta(\xi)=1\ \text{\textrm{and}}\ d\eta(\xi,X)=0.
\]
In fact, the tensor field of type $(1,1)$, defined by $\Phi(Y)=\nabla_{Y}\xi$,
satisfies the condition%
\[
(\nabla_{X}\Phi)(Y)=g(\xi,Y)X-g(X,Y)\xi
\]
for any pair of vector fields $X$ and $Y$ on $M$. Then such a triple
$(\eta,\xi,\Phi)$ is called a Sasakian structure on a Sasakian manifold
$(M,g).$ Note that the Riemannian curvature satisfying the following
\[
R(X,\xi)Y=g(\xi,Y)X-g(X,Y)\xi
\]
for any pair of vector fields $X$ and $Y$ on $M$. In particular, the sectional
curvature of every section containing $\xi$ equals one.

\subsection{Foliation Normal Local Coordinate}

Let $(M,\eta,\xi,\Phi,g)$ be a compact Sasakian $(2n+1)$-manifold with
$g(\xi,\xi)=1$ and the integral curves of $\xi$ are geodesics. For any point
$p\in M$, we can construct local coordinates in a neighborhood of $p$ which
are simultaneously foliated and Riemann normal coordinates (\cite{gkn}). That
is, we can find Riemann normal coordinates $\{x,z^{1},z^{2},\cdots,z^{n}\}$ on
a neighborhood $U$ of $p$, such that $\frac{\partial}{\partial x}=\xi$ on $U$.
Let $\{U_{\alpha}\}_{\alpha\in A}$ be an open covering of the Sasakian
manifold and
\[
\pi_{\alpha}:U_{\alpha}\rightarrow V_{\alpha}\subset%
\mathbb{C}
^{n}%
\]
submersion such that $\pi_{\alpha}\circ\pi_{\beta}^{-1}:\pi_{\beta}(U_{\alpha
}\cap U_{\beta})\rightarrow\pi_{\alpha}(U_{\alpha}\cap U_{\beta})$ is
biholomorphic. On each $V_{\alpha},$ there is a canonical isomorphism
\[
d\pi_{\alpha}:D_{p}\rightarrow T_{\pi_{\alpha}(p)}V_{\alpha}%
\]
for any $p\in U_{\alpha},$ where $D=\ker\xi\subset TM.$ Since $\xi$ generates
isometrics, the restriction of the Sasakian metric $g$ to $D$ gives a
well-defined Hermitian metric $g_{\alpha}^{T}$ on $V_{\alpha}.$ This Hermitian
metric in fact is K\"{a}hler. More precisely, let $z^{1},z^{2},\cdots,z^{n}$
be the local holomorphic coordinates on $V_{\alpha}$. We pull back these to
$U_{\alpha}$ and still write the same. Let $x$ be the coordinate along the
leaves with $\xi=\frac{\partial}{\partial x}.$ Then we have the foliation
local coordinate $\{x,z^{1},z^{2},\cdots,z^{n}\}$ on $U_{\alpha}\ $and
$(D\otimes%
\mathbb{C}
)$ is spanned by the form
\[%
\begin{array}
[c]{c}%
Z_{\alpha}=\left(  \frac{\partial}{\partial z^{\alpha}}-\theta\left(
\frac{\partial}{\partial z^{\alpha}}\right)  \frac{\partial}{\partial
x}\right)  ,\ \alpha=1,2,\cdots,n.
\end{array}
\]
Moreover%
\[%
\begin{array}
[c]{c}%
\Phi=\sqrt{-1}\left(  \frac{\partial}{\partial z^{j}}+\sqrt{-1}h_{j}%
\frac{\partial}{\partial x}\right)  \otimes dz^{j}+conj
\end{array}
\]
and
\[%
\begin{array}
[c]{c}%
\eta=dx-\sqrt{-1}h_{j}dz^{j}+\sqrt{-1}h_{\overline{j}}d\overline{z}^{j}.
\end{array}
\]
Here $h$ is basic: $\frac{\partial h}{\partial x}=0$ and $h_{j}=\frac{\partial
h}{\partial z^{j}},h_{j\overline{l}}=\frac{\partial^{2}h}{\partial
z^{j}\partial\overline{z}^{l}}$ with the normal coordinate
\begin{equation}
h_{j}(p)=0,\text{ }h_{j\overline{l}}(p)=\delta_{j}^{l},\text{ }dh_{j\overline
{l}}(p)=0.\label{AAA3}%
\end{equation}
A frame
\[%
\begin{array}
[c]{c}%
\{\frac{\partial}{\partial x},\text{ }Z_{j}=\left(  \frac{\partial}{\partial
z^{j}}+\sqrt{-1}h_{j}\frac{\partial}{\partial x}\right)  ,\ j=1,2,\cdots,n\}
\end{array}
\]
and the dual
\[
\{\eta,dz^{j},\ j=1,2,\cdots,n\}
\]
with
\[
\lbrack Z_{i},Z_{j}]=[\xi,Z_{j}]=0.
\]

Since $i(\xi)d\eta=0,$
\[%
\begin{array}
[c]{c}%
d\eta(Z_{\alpha},\overline{Z_{\beta}})=d\eta(\frac{\partial}{\partial
z^{\alpha}},\frac{\partial}{\overline{\partial}z^{\beta}}).
\end{array}
\]
Then the K\"{a}hler $2$-form $\omega_{\alpha}^{T}$ of the Hermitian metric
$g_{\alpha}^{T}$ on $V_{\alpha},$ which is the same as the restriction of the
Levi form $d\eta$ to $\widetilde{D_{\alpha}^{n}}$, the slice $\{x=$
\textrm{constant}$\}$ in $U_{\alpha},$ is closed. The collection of K\"{a}hler
metrics $\{g_{\alpha}^{T}\}$ on $\{V_{\alpha}\}$ is so-called a transverse
K\"{a}hler metric. We often refer to $d\eta$ as the K\"{a}hler form of the
transverse K\"{a}hler metric $g^{T}$ in the leaf space $\widetilde{D^{n}}.$

The K\"{a}hler form $d\eta$ on $D$ and the K\"{a}hler metric $g^{T}$ is define
such that
\[
g=g^{T}+\eta\otimes\eta
\]
and
\[%
\begin{array}
[c]{c}%
g_{i\overline{j}}^{T}=g^{T}(\frac{\partial}{\partial z^{i}},\frac{\partial
}{\partial\overline{z}^{j}})=d\eta(\frac{\partial}{\partial z^{i}},\Phi
\frac{\partial}{\partial\overline{z}^{j}})=2h_{i\overline{j}}.
\end{array}
\]
In terms of the normal coordinate, we have%
\[%
\begin{array}
[c]{c}%
g^{T}=g_{i\overline{j}}^{T}dz^{i}d\overline{z}^{j},\text{ }\omega=d\eta
=2\sqrt{-1}h_{i\overline{j}}dz^{i}\wedge d\overline{z}^{j}.
\end{array}
\]

The transverse Ricci curvature $\mathrm{Ric}^{T}$ of the Levi-Civita
connection $\nabla^{T}$ associated to $g^{T}$ is
\[
\mathrm{Ric}^{T}=\mathrm{Ric}+2g^{T}%
\]
and%
\[
R^{T}=R+2n.
\]
The transverse Ricci form $\rho^{T}$
\[%
\begin{array}
[c]{c}%
\rho^{T}=\mathrm{Ric}^{T}(J\cdot,\cdot)=-\sqrt{-1}R_{i\overline{j}}^{T}%
dz^{i}\wedge d\overline{z}^{j}%
\end{array}
\]
with
\[%
\begin{array}
[c]{c}%
R_{i\overline{j}}^{T}=-\frac{\partial_{2}}{\partial z^{i}\partial\overline
{z}^{j}}\log\det(g_{\alpha\overline{\beta}}^{T})
\end{array}
\]
and it is a closed basic $(1,1)$-form
\[
\rho^{T}=\rho+2d\eta.
\]

\subsection{Basic Cohomology and Type II Deformation in a Sasakian Manifold}

\begin{definition}
Let $(M,\eta,\xi,\Phi,g)$ be a Sasakian $(2n+1)$-manifold. Define a $p$-form
$\gamma$ is called basic if%
\[
i(\xi)\gamma=0\text{ and }\mathcal{L}_{\xi}\gamma=0.
\]

\end{definition}

Let $\Lambda_{B}^{p}$ be the sheaf of germs of basic $p$-forms and $\Omega
_{B}^{p}$ be the set of all global sections of $\Lambda_{B}^{p}$. It is easy
to check that $d\gamma$ is basic if $\gamma$ is basic. Set $d_{B}%
=d|_{\Omega_{B}^{p}},$ then%
\[
d_{B}:\Omega_{B}^{p}\rightarrow\Omega_{B}^{p+1}.
\]
We then have the well-defined operators
\[
d_{B}:=\partial_{B}+\overline{\partial}_{B}%
\]
with
\[
\partial_{B}:\Lambda_{B}^{p,q}\rightarrow\Lambda_{B}^{p+1,q}%
\]
and
\[
\overline{\partial}_{B}:\Lambda_{B}^{p,q}\rightarrow\Lambda_{B}^{p,q+1}.
\]
Then for $d_{B}^{c}:=\frac{1}{2}\sqrt{-1}(\overline{\partial}_{B}-\partial
_{B}),$ we have%
\[
d_{B}d_{B}^{c}=\sqrt{-1}\partial_{B}\overline{\partial}_{B},\text{ }d_{B}%
^{2}=(d_{B}^{c})^{2}=0.
\]
The basic Laplacian is defined by
\[
\Delta_{B}:=d_{B}d_{B}^{\ast}+d_{B}^{\ast}d_{B}.
\]
Then we have the basic de Rham complex $(\Omega_{B}^{\ast},d_{B})$ and the
basic Dolbeault complex $(\Omega_{B}^{p,\ast},\overline{\partial}_{B})$ and
its cohomology group\ $H_{B}^{\ast}(M,\mathbb{R})$ (\cite{eka}).

\begin{definition}
(i) We define the basic cohomology of the foliation $F_{\xi}$by
\[
H_{B}^{\ast}(F_{\xi}):=H_{B}^{\ast}(M,\mathbb{R}).
\]
Then by transverse Hodge decomposition and transverse Serre duality
\[
H_{B}^{p,q}(F_{\xi})\simeq H_{B}^{q,p}(F_{\xi})
\]
and the cohomology of the leaf space $Z=M/U(1)$ to this basic cohomology of
the foliation
\[
H_{orb}^{\ast}(Z,\mathbb{R})=H_{B}^{\ast}(F_{\xi}):=H_{B}^{\ast}%
(M,\mathbb{R}).
\]

(ii) Define the basic first Chern class $c_{1}^{B}(M)$ by
\[
c_{1}^{B}=[\rho^{T}]_{B}%
\]
and the transverse Einstein (Sasaki $\eta$-Einstein) equation up to a
$D$-homothetic deformation
\[
\lbrack\rho^{T}]_{B}=\varkappa\lbrack d\eta]_{B},\text{ }\varkappa=-1,0,1.
\]
Basic $k$-th Chern class $c_{k}^{B}(M)$ is represented by a closed basic
$(k,k)$-form $\gamma_{k}$ which is determined by the formula%
\[%
\begin{array}
[c]{c}%
\det\left(  I_{n}+\frac{\sqrt{-1}}{2\pi}\Omega^{T}\right)  =1+\gamma
_{1}+\cdots+\gamma_{k}.
\end{array}
\]
Here $\Omega^{T}$ is the curvature $2$-form of type basic $(1,1)$ with respect
to the transverse connection $\nabla^{T}$.
\end{definition}

\begin{definition}
We define
Type II deformations of Sasakian structures %
$(M,\eta,\xi,\Phi,g)$ by fixing the $\xi$ and varying $\eta$. That is, for
$\varphi\in\Omega_{B}^{0}$, define
\[
\widetilde{\eta}=\eta+d_{B}^{c}\varphi,
\]
then
\[%
\begin{array}
[c]{c}%
d\widetilde{\eta}=d\eta+d_{B}d_{B}^{c}\varphi=d\eta+\sqrt{-1}\partial
_{B}\overline{\partial}_{B}\varphi
\end{array}
\]
and
\[%
\begin{array}
[c]{c}%
\widetilde{\omega}=\omega+\sqrt{-1}\partial_{B}\overline{\partial}_{B}\varphi.
\end{array}
\]

\end{definition}

Note that
we have %
the same transversal holomorphic foliation ($\xi$ is fixed) but with the new
K\"{a}hler structure on the K\"{a}hler cone $C(M)$ and new contact bundle
$\widetilde{D}$: $\widetilde{\omega}=dd^{c}\widetilde{r},$ $\widetilde{r}%
=re^{\varphi}$. The same holomorphic structure: $r\frac{\partial}{\partial
r}=\widetilde{r}\frac{\partial}{\partial\widetilde{r}};$ $\xi=J(r\frac
{\partial}{\partial r})$ and $\xi+\sqrt{-1}r\frac{\partial}{\partial r}%
=\xi-\sqrt{-1}\Phi(\xi)$ is the holomorphic vector field on $C(M).$ Moreover,
we have%
\[%
\begin{array}
[c]{l}%
\widetilde{\Phi}=\Phi-\xi\otimes(d_{B}^{c}\varphi)\circ\Phi,\\
\widetilde{g}=d\widetilde{\eta}\circ(Id\otimes\widetilde{\Phi}%
)+\widetilde{\eta}\otimes\widetilde{\eta}.
\end{array}
\]
and
\[
\mathcal{L}_{\xi}\widetilde{\Phi}=\mathcal{L}_{\xi}\Phi=0.
\]

\section{Asymptotic Convergence of the Sasaki-Ricci Flow}

In this section, we will establish the Sasaki analogue of asymptotic
convergence of solutions of the K\"{a}hler-Ricci flow which is the starting
step to prove the main theorem in this paper.

\subsection{The Sasaki-Ricci Flow}

By a $\partial_{B}\overline{\partial}_{B}$-Lemma (\cite{eka}) in the basic
Hodge decomposition, there is a basic function $F:M\rightarrow\mathbb{R}$ such
that
\[
\rho^{T}(x,t)-\varkappa d\eta(x,t)=d_{B}d_{B}^{c}F=\sqrt{-1}\partial
_{B}\overline{\partial}_{B}F.
\]
We focus on finding a new $\eta$-Einstein Sasakian structure $(M,\xi
,\widetilde{\eta},\widetilde{\Phi},\widetilde{g})$ with
\[
\widetilde{\eta}=\eta+d_{B}^{c}\varphi,\text{ }\varphi\in\Omega_{B}^{0}%
\]
and
\[%
\begin{array}
[c]{c}%
\widetilde{g}^{T}=(g_{i\overline{j}}^{T}+\varphi_{i\overline{j}})dz^{i}\wedge
d\overline{z}^{j}=2\sqrt{-1}(h_{i\overline{j}}+\frac{1}{2}\varphi
_{i\overline{j}})dz^{i}\wedge d\overline{z}^{j}%
\end{array}
\]
such that
\[
\widetilde{\rho}^{T}=\varkappa d\widetilde{\eta}.
\]
Hence
\[
\widetilde{\rho}^{T}-\rho^{T}=\kappa d_{B}d_{B}^{c}\varphi-d_{B}d_{B}^{c}F
\]
and it follows\ that
\begin{equation}%
\begin{array}
[c]{c}%
\frac{\det(g_{\alpha\overline{\beta}}^{T}+\varphi_{\alpha\overline{\beta}}%
)}{\det(g_{\alpha\overline{\beta}}^{T})}=e^{-\kappa\varphi+F}.
\end{array}
\label{B}%
\end{equation}
This is a Sasakian analogue of the Monge-Ampere equation for the orbifold
version of Calabi-Yau Theorem (\cite{eka}).

Now we consider the Sasaki-Ricci flow on $M\times\lbrack0,T)$%
\begin{equation}%
\begin{array}
[c]{c}%
\frac{d}{dt}g^{T}(x,t)=-\mathrm{Ric}^{T}(x,t)+\varkappa g^{T}(x,t)
\end{array}
\label{2021}%
\end{equation}
or
\[%
\begin{array}
[c]{c}%
\frac{d}{dt}d\eta(x,t)=-\rho^{T}(x,t)+\varkappa d\eta(x,t).
\end{array}
\]
It is equivalent to consider%
\begin{equation}%
\begin{array}
[c]{c}%
\frac{d}{dt}\varphi=\log\det(g_{\alpha\overline{\beta}}^{T}+\varphi
_{\alpha\overline{\beta}})-\log\det(g_{\alpha\overline{\beta}}^{T}%
)+\kappa\varphi-F.
\end{array}
\label{C}%
\end{equation}
Note that, for any two Sasakian structures with the fixed Reeb vector field
$\xi,$ we have
\[
\mathrm{Vol}(M,g)=\mathrm{Vol}(M,g^{\prime})
\]
and
\[%
\begin{array}
[c]{c}%
\widetilde{\omega}^{n}\wedge\eta=(\sqrt{-1})^{n}\det(g_{\alpha\overline{\beta
}}^{T}+\varphi_{\alpha\overline{\beta}})dz^{1}\wedge d\overline{z}^{1}%
\wedge\cdots\wedge dz^{n}\wedge d\overline{z}^{n}\wedge dx.
\end{array}
\]

\subsection{Convergence of Solutions of the Sasaki-Ricci Flow}

Let $(M,\eta_{0},\xi_{0},\Phi_{0},g_{0},\omega_{0})$ be a compact
quasi-regular Sasakian $(2n+1)$-manifold and its the space $Z$ of leaves of
the characteristic foliation be well-formed\textbf{. }We consider\textbf{\ }a
solution $\omega=\omega(t)$ of the Sasaki-Ricci flow%
\begin{equation}%
\begin{array}
[c]{c}%
\frac{\partial}{\partial t}\omega(t)=-\mathrm{Ric}_{\omega(t)}^{T},\text{
}\omega(0)=\omega_{0}.
\end{array}
\label{1}%
\end{equation}
As long as the solution exists, the cohomology class $[\omega(t)]_{B}$ evolves
by%
\[%
\begin{array}
[c]{c}%
\frac{\partial}{\partial t}\left[  \omega(t)\right]  _{B}=-c_{1}^{B}(M),\text{
}\left[  \omega(0)\right]  _{B}=\left[  \omega_{0}\right]  _{B},
\end{array}
\]
and solving this ordinary differential equation gives%
\[%
\begin{array}
[c]{c}%
\left[  \omega(t)\right]  _{B}=\left[  \omega_{0}\right]  _{B}-tc_{1}^{B}(M).
\end{array}
\]
We see that a necessary condition for the Sasaki-Ricci flow to exist for $t>0
$ such that%
\[%
\begin{array}
[c]{c}%
\left[  \omega_{0}\right]  _{B}-tc_{1}^{B}(M)>0.
\end{array}
\]
This necessary condition is in fact sufficient. In fact we define
\[%
\begin{array}
[c]{c}%
T_{0}:=\sup\{t>0|\text{ }\left[  \omega_{0}\right]  _{B}-tc_{1}^{B}(M)>0\}.
\end{array}
\]
That is to say that
\begin{equation}%
\begin{array}
[c]{c}%
\left[  \omega_{0}\right]  _{B}-T_{0}c_{1}^{B}(M)\in\overline{C_{M}^{B}}%
\end{array}
\label{E}%
\end{equation}
which is a nef class. Here
\[%
\begin{array}
[c]{c}%
C_{M}^{B}=\{[\alpha]_{B}\in H_{B}^{1,1}(M,\mathbb{R})|\text{ }\exists\text{
transverse K\"{a}hler metric }\omega>0\text{ such that }[\omega]_{B}%
=[\alpha]_{B}\}.
\end{array}
\]

For a representative $\chi\in-c_{1}^{B}(M)$, we can fix a transverse volume
form $\Omega$ on $(M,\xi_{0},\eta_{0},\Phi_{0},g_{0},\omega_{0})$ such that
\[%
\begin{array}
[c]{c}%
\Omega\wedge\eta_{0}=(\sqrt{-1})^{n}F(z_{1},\cdots,z_{n})dz_{1}\wedge
d\overline{z}_{1}\wedge\cdots\wedge dz_{n}\wedge d\overline{z}_{n}\wedge dx
\end{array}
\]
with
\[%
\begin{array}
[c]{c}%
\sqrt{-1}\partial_{B}\overline{\partial}_{B}\log F=-\mathrm{Ric}^{T}%
(\Omega)=\chi
\end{array}
\]
and \
\[%
\begin{array}
[c]{c}%
\int_{M}\Omega\wedge\eta_{0}=\int_{M}\omega_{0}^{n}\wedge\eta_{0}.
\end{array}
\]
We choose a reference (transverse) K\"{a}hler metric%
\[%
\begin{array}
[c]{c}%
\widehat{\omega}_{t}:=\omega_{0}+t\chi.
\end{array}
\]
Then the corresponding transverse parabolic complex Monge-Ampere equation to
(\ref{1}) on $M\times\lbrack0,T_{0})$ is%
\begin{equation}
\left\{
\begin{array}
[c]{rll}%
\frac{\partial}{\partial t}\varphi(x,t) & = & \log\frac{(\widehat{\omega}%
_{t}+\sqrt{-1}\partial_{B}\overline{\partial}_{B}\varphi)^{n}\wedge\eta_{0}%
}{\Omega\wedge\eta_{0}},\\
\widehat{\omega}_{t} & = & \omega_{0}+t\chi.\\
\sqrt{-1}\partial_{B}\overline{\partial}_{B}\log\Omega & = & \chi,\\
\widehat{\omega}_{t}+\sqrt{-1}\partial_{B}\overline{\partial}_{B}\varphi & > &
0,\\
\varphi(0) & = & 0.
\end{array}
\right. \label{2}%
\end{equation}

Based on \cite{sw1}, \cite{t} and references therein as in the K\"{a}hler
case, we have the following cohomological characterization for the maximal
solution of the Sasaki-Ricci flow (we refer to the proof of Theorem \ref{T32}
as below):

\begin{theorem}
There exists a unique maximal solution $\omega(t)$ of the Sasaki-Ricci flow
(\ref{1}) on $M$ for $t\in\lbrack0,T_{0})$.
\end{theorem}

Next if we assume that $K_{M}^{T}$ is nef, it follows from (\ref{D}) and
(\ref{E}) that $T_{0}=\infty.$ On the other hand if $K_{M}^{T}$ is big also%
\[%
\begin{array}
[c]{c}%
\int_{M}(c_{1}^{B}(K_{M}^{T}))^{n}\wedge\eta_{0}>0.
\end{array}
\]
Now by Sasaki analogue of Kawamata base-point free theorem (\cite{clw}), we
obtain that $K_{M}^{T}$ is semi-ample, then there exists a $S^{1}$-equivariant
basic base-point free transverse holomorphic map
\[%
\begin{array}
[c]{c}%
\Psi:M\rightarrow(\mathbb{C}\mathrm{P}^{N},\omega_{FS})
\end{array}
\]
defined by the basic transverse holomorphic section $\{s_{0},s_{1}%
,\cdots,s_{N}\}$ of $H^{0}(M,(K_{M}^{T})^{m})$ which is $S^{1}$-equivariant
with respect to the weighted $\mathbb{C}^{\ast}$action. Here $N=\dim
H^{0}(M,(K_{M}^{T})^{m})-1$ for a large positive integer $m$ and
\[%
\begin{array}
[c]{c}%
\frac{1}{m}\Psi^{\ast}(\omega_{FS})=\widehat{\omega}_{\infty}\geq0.
\end{array}
\]

For the asymptotic behavior, we need to rescale the Sasaki-Ricci flow
(\ref{1}) to the normalized Sasaki-Ricci flow on $M\times\lbrack0,\infty)$ as
following:%
\begin{equation}
\left\{
\begin{array}
[c]{rll}%
\frac{\partial}{\partial t}\varphi(x,t) & = & \log\frac{(\widehat{\omega}%
_{t}+\sqrt{-1}\partial_{B}\overline{\partial}_{B}\varphi)^{n}\wedge\eta_{0}%
}{\Omega\wedge\eta_{0}}-\varphi,\\
\widehat{\omega}_{t} & = & e^{-t}\omega_{0}+(1-e^{-t})\widehat{\omega}%
_{\infty},\\
\sqrt{-1}\partial_{B}\overline{\partial}_{B}\log\Omega & = & \widehat{\omega
}_{\infty},\\
\widehat{\omega}_{t}+\sqrt{-1}\partial_{B}\overline{\partial}_{B}\varphi & > &
0,\\
\varphi(0) & = & 0.
\end{array}
\right. \label{3}%
\end{equation}
We observe that $\chi=\widehat{\omega}_{\infty}\in-c_{1}^{B}$ is a nonnegative
$(1,1)$-current K\"{a}hler metric. The starting point to show Theorem
\ref{T11}, we must be able to derive the following basic result which was
established for the K\"{a}hler-Ricci flow due to \cite{tsu}, \cite{t}, and
\cite{tz3}.

\begin{theorem}
\label{T32} Let $(M,\eta_{0},\xi_{0},\Phi_{0},g_{0},\omega_{0})$ be a compact
quasi-regular Sasakian $(2n+1)$-manifold and its the space $Z$ of leaves of
the characteristic foliation be well-formed\textbf{.} Suppose that $-c_{1}%
^{B}(M)\in\overline{C_{M}^{B}}$ and
\[%
\begin{array}
[c]{c}%
\int_{M}(-c_{1}^{B}(M))^{n}\wedge\eta_{0}>0.
\end{array}
\]
Then there exists a unique solution $\omega(t)$ of the Sasaki-Ricci flow
(\ref{1}) on $M\times\lbrack0,\infty).$ Furthermore, there exists an $\eta
$-Einstein metric $\omega_{\infty}$ on $M\backslash\mathrm{Null}(-c_{1}%
^{B}(M))$ which satisfies
\[%
\begin{array}
[c]{c}%
\mathrm{Ric}_{\omega_{\infty}}^{T}=-\omega_{\infty}%
\end{array}
\]
such that for any initial transverse K\"{a}hler metric $\omega_{0}$, the
rescaled metrics $\frac{\omega(t)}{t}$ converge smoothly on compact subsets of
$M\backslash\mathrm{Null}(-c_{1}^{B}(M))$ to $\omega_{\infty}$ as
$t\rightarrow\infty$.
\end{theorem}

We first state the Sasaki analogue of Kodaira lemma for the further
application. By the first structure theorem (Proposition \ref{P21}), $M$\ is a
principal $S^{1}$-orbibundle over a Hodge orbifold $Z$ which is also a
$Q$-factorial, polarized, normal projective variety such that there is a
Riemannian submersion$\ $%
\[
\pi:(M,g_{0},\omega_{0})\rightarrow(Z,h_{0},\omega_{h_{0}})
\]
such that $\omega_{0}=\pi^{\ast}(\omega_{h_{0}})$. We define
\[%
\begin{array}
[c]{c}%
\mathrm{Null}(\alpha\mathbf{)=\cup}_{\int_{V}\alpha^{\frac{\dim V-1}{2}}%
\wedge\eta_{0}=0}V
\end{array}
\]
which is the union over all positive-dimensional invariant $(2m-1)$%
-submanifolds $V\subset M$ such that $\pi(V)=E\subset Z$ and $\pi^{\ast
}(\widehat{\alpha})=\alpha$ so that
\[%
\begin{array}
[c]{c}%
\int_{E}\widehat{\alpha}^{\frac{\dim E-1}{2}}=0.
\end{array}
\]
By applying the arguments as Collins-Tosatti \cite{ct} and Demailly-Paun
\cite{dp} (also \cite{d}) to a Hodge orbifold $Z$ which is a normal projective
variety and lifting to $M$ via the Riemannian submersion $\pi,$ it follows that

\begin{proposition}
\label{P31} Let $(M,\eta_{0},\xi_{0},\Phi_{0},g_{0},\omega_{0})$ be a compact
quasi-regular Sasakian $(2n+1)$-manifold and $\alpha$ a basic closed real
$(1,1)$-form whose class $[\alpha]_{B}$ is nef and big
\[%
\begin{array}
[c]{c}%
\int_{M}\alpha^{n}\wedge\eta_{0}>0.
\end{array}
\]
Then there exists an upper semicontinuous $L^{1}$-function $\phi
:M\rightarrow\mathbb{R}\cup\{-\infty\},$ with $\sup_{M}\phi=0$ which is basic
and equals $-\infty$ on $\mathrm{Null}(\alpha\mathbf{)}$ and is finite, smooth
on $M\backslash\mathrm{Null}(\alpha\mathbf{)}$ such that
\[%
\begin{array}
[c]{c}%
\alpha+\sqrt{-1}\partial_{B}\overline{\partial}_{B}\phi\geq\varepsilon
\omega_{0}%
\end{array}
\]
on $M\backslash\mathrm{Null}(\alpha\mathbf{)}$, for some $\varepsilon>0.$
\end{proposition}

We first show the following uniform estimate.

\begin{lemma}
There exists $C>0$ such that on $M\times\lbrack0,\infty),$ we have%
\begin{equation}%
\begin{array}
[c]{c}%
|\varphi|+|\frac{\partial\varphi}{\partial t}|^{2}\leq C.
\end{array}
\label{6}%
\end{equation}

\end{lemma}

\begin{proof}
Let $\omega=\omega(t)$ be the solution to the normalized Sasaki-Ricci flow
(\ref{3}). First, we show that
\begin{equation}%
\begin{array}
[c]{c}%
\varphi(t)\leq C
\end{array}
\label{11}%
\end{equation}
on $M\times\lbrack0,\infty)$. This is a simple consequence of the maximum
principle since at any maximum point of $\varphi$ (for $t>0$) we have%
\[%
\begin{array}
[c]{c}%
0\leq\frac{\partial\varphi}{\partial t}=\log\frac{(\widehat{\omega}_{t}%
+\sqrt{-1}\partial_{B}\overline{\partial}_{B}\varphi(t))^{n}\wedge\eta_{0}%
}{\Omega\wedge\eta_{0}}-\varphi(t)\leq\log\frac{\widehat{\omega}_{t}{}%
^{n}\wedge\eta_{0}}{\Omega\wedge\eta_{0}}-\varphi(t)\leq C-\varphi(t),
\end{array}
\]
using that at a maximum point $\widehat{\omega}_{t}\geq\widehat{\omega}%
_{t}+\sqrt{-1}\partial_{B}\overline{\partial}_{B}\varphi(t)>0$, and we are
done. Next, we show that%
\begin{equation}%
\begin{array}
[c]{c}%
\frac{\partial\varphi}{\partial t}=\varphi^{\prime}(t)\leq C(1+t)e^{-t},
\end{array}
\label{12}%
\end{equation}
on $M\times\lbrack0,\infty)$. Indeed we compute
\[%
\begin{array}
[c]{c}%
(\frac{\partial}{\partial t}-\Delta_{B})\varphi(t)=\varphi^{\prime
}(t)-n+tr_{\omega}\widehat{\omega}_{t},
\end{array}
\]%
\[%
\begin{array}
[c]{c}%
(\frac{\partial}{\partial t}-\Delta_{B})\varphi^{\prime}(t)=-\varphi^{\prime
}(t)-e^{-t}tr_{\omega}\omega_{0}+e^{-t}tr_{\omega}\widehat{\omega}_{\infty},
\end{array}
\]%
\[%
\begin{array}
[c]{c}%
(\frac{\partial}{\partial t}-\Delta_{B})((e^{t}-1)\varphi^{\prime}%
(t)-\varphi(t)-nt)=-tr_{\omega}\omega_{0}<0,
\end{array}
\]
and so the maximum principle gives%
\[%
\begin{array}
[c]{c}%
(e^{t}-1)\varphi^{\prime}(t)-\varphi(t)-nt\leq0
\end{array}
\]
which together with (\ref{11}) gives (\ref{12}) for $t>1.$On the other hand,
it is clear that (\ref{12}) holds for $0\leq t\leq1$.

Next, we show that there is a constant $C>0$ such that%
\begin{equation}%
\begin{array}
[c]{c}%
\varphi^{\prime}(t)+\varphi(t)\geq\phi-C,
\end{array}
\label{13}%
\end{equation}
on $M\times\lbrack0,\infty),$ here we apply Proposition \ref{P31} to obtain an
upper semicontinuous $L^{1}$ function $\phi:M\rightarrow\mathbb{R}%
\cup\{-\infty\}$, with $\sup_{M}\phi=0$ and equals $-\infty$ on $\mathrm{Null}%
(-c_{1}^{B}(M))$, which is finite and smooth on $M\backslash\mathrm{Null}%
(-c_{1}^{B}(M))$ such that%
\begin{equation}%
\begin{array}
[c]{c}%
\chi+\sqrt{-1}\partial_{B}\overline{\partial}_{B}\phi\geq\varepsilon\omega
_{0},
\end{array}
\label{4}%
\end{equation}
on $M\backslash\mathrm{Null}(-c_{1}^{B}(M))$, for some $\varepsilon>0.$
Consider the quantity%
\[%
\begin{array}
[c]{c}%
Q=\varphi^{\prime}(t)+\varphi(t)-\phi.
\end{array}
\]
The function $Q$ is lower semicontinuous and it approaches $\infty$ as we
approach $\mathrm{Null}(-c_{1}^{B}(M)),$ and so it achieves a minimum at
$(x,t)$, for some $t>0$ and $x\notin\mathrm{Null}(-c_{1}^{B}(M)),$ and at this
point we have%
\[%
\begin{array}
[c]{lll}%
0 & \geq & (\frac{\partial}{\partial t}-\Delta_{B})Q=tr_{\omega}%
(\widehat{\omega}_{\infty}+\sqrt{-1}\partial_{B}\overline{\partial}_{B}%
\phi)-n\geq\varepsilon tr_{\omega}\omega_{0}-n\\
& \geq & n\varepsilon\left(  \frac{\omega_{0}^{n}\wedge\eta_{0}}{\omega
^{n}\wedge\eta_{0}}\right)  ^{\frac{1}{n}}-n\geq C^{-1}e^{-\frac
{\varphi^{\prime}(t)+\varphi(t)}{n}}-n,
\end{array}
\]
and so
\[%
\begin{array}
[c]{c}%
\varphi^{\prime}(t)+\varphi(t)\geq-C,
\end{array}
\]
which implies that
\[%
\begin{array}
[c]{c}%
Q\geq-C
\end{array}
\]
since $\phi\leq0,$ this shows (\ref{13}).
\end{proof}

In the following, we prove Theorem \ref{T32}:

\begin{proof}
The proof is similar to the K\"{a}hler-Ricci flow case. Since the convergence
is for the rescaled metrics $\frac{\omega(t)}{t},$ it is convenient to
renormalize the Sasaki-Ricci flow as follows:%
\begin{equation}%
\begin{array}
[c]{c}%
\frac{\partial}{\partial t}\omega(t)=-\mathrm{Ric}_{\omega(t)}^{T}%
-\omega(t),\text{ }\omega(0)=\omega_{0}.
\end{array}
\label{14a}%
\end{equation}
Then the goal is to show that the solution $\omega(t)$ of (\ref{14a})
satisfies\
\begin{equation}%
\begin{array}
[c]{c}%
\omega(t)\rightarrow\omega_{\infty},
\end{array}
\label{14b}%
\end{equation}
in $C_{\mathrm{loc}}^{\infty}(M\backslash\mathrm{Null}(-c_{1}^{B}(M))$ as
$t\rightarrow\infty$ and that the limit $\omega_{\infty}$ is transverse K\"{a}hler-Einstein.

Note that (\ref{14a}) is equivalent to (\ref{3}) and recall that, from
(\ref{6}), we have the uniform bounded estimates for $\varphi(t)$ and
$\varphi^{\prime}(t)$ on $M\times\lbrack0,\infty).$ First we show that there
is a constant $C>0$ such that%
\begin{equation}%
\begin{array}
[c]{c}%
tr_{\omega_{0}}\omega(t)\leq Ce^{-C\phi},
\end{array}
\label{14}%
\end{equation}
on $M\times\lbrack0,\infty),$ here $\phi$ is the function in Proposition
\ref{P31} for $\alpha=\chi$. Like the K\"{a}hler case, we can obtain
\[%
\begin{array}
[c]{c}%
(\frac{\partial}{\partial t}-\Delta_{B})\log tr_{\omega_{0}}\omega(t)\leq
Ctr_{\omega(t)}\omega_{0},
\end{array}
\]
where $-C$ is a lower bound for the transverse bisectional curvature with
respect to $\omega_{0}$. Using this inequality and (\ref{4}), we compute
\[%
\begin{array}
[c]{l}%
(\frac{\partial}{\partial t}-\Delta_{B})(\log tr_{\omega_{0}}\omega
(t)-A(\varphi^{\prime}(t)+\varphi(t)-\phi))\\
\leq Ctr_{\omega(t)}\omega_{0}+An-Atr_{\omega(t)}(\chi+\sqrt{-1}\partial
_{B}\overline{\partial}_{B}\phi)\\
\leq-tr_{\omega(t)}\omega_{0}+C,
\end{array}
\]
on $M\backslash\mathrm{Null}(-c_{1}^{B}(M)),$ if we choose a positive $A$
large enough such that $C\leq A\varepsilon-1.$ Therefore at a maximum point
$(x,t)$ of this quantity for some $t>0$ with $x\notin\mathrm{Null}(-c_{1}%
^{B}(M)),$ we have%
\[%
\begin{array}
[c]{c}%
tr_{\omega(t)}\omega_{0}\leq C.
\end{array}
\]
By applying the inequality
\[%
\begin{array}
[c]{c}%
tr_{\omega_{0}}\omega(t)\leq\frac{(tr_{\omega(t)}\omega_{0})^{n-1}}%
{(n-1)!}\frac{\omega(t)^{n}\wedge\eta_{0}}{\omega_{0}^{n}\wedge\eta_{0}},
\end{array}
\]
we conclude that at the maximum point $(x,t)$ we have
\[%
\begin{array}
[c]{c}%
tr_{\omega_{0}}\omega(t)\leq C\frac{\omega(t)^{n}\wedge\eta_{0}}{\omega
_{0}^{n}\wedge\eta_{0}}=Ce^{\varphi^{\prime}(t)+\varphi(t)}\frac{\Omega
\wedge\eta_{0}}{\omega_{0}^{n}\wedge\eta_{0}}\leq C,
\end{array}
\]
used the estimate (\ref{6}). Combining this with (\ref{13}) it implies
\[%
\begin{array}
[c]{l}%
\log tr_{\omega_{0}}\omega(t)-A(\varphi^{\prime}(t)+\varphi(t)-\phi)\leq C,
\end{array}
\]
at the maximum and hence everywhere, and thus yields (\ref{14}). Also note
that
\[%
\begin{array}
[c]{c}%
\frac{\omega(t)^{n}\wedge\eta_{0}}{\omega_{0}^{n}\wedge\eta_{0}}\geq
C^{-1}e^{\varphi^{\prime}(t)+\varphi(t)}\geq C^{-1}e^{-\phi},
\end{array}
\]
from (\ref{13}) and so given any compact subset $K\subset M\backslash
\mathrm{Null}(-c_{1}^{B}(M))$ there exists a constant $C_{K}$ such that%
\begin{equation}%
\begin{array}
[c]{c}%
C_{K}^{-1}\omega_{0}\leq\omega(t)\leq C_{K}\omega_{0}%
\end{array}
\label{15}%
\end{equation}
on $K\times\lbrack0,\infty).$ The higher order estimate on $K$ is then given
by
\[%
\begin{array}
[c]{c}%
||\omega(t)||_{C^{k}(K,\omega_{0})}\leq C_{K,k},
\end{array}
\]
for all $t\geq0,$ $k\geq0,$ up to shrinking $K$ slightly. These estimates will
imply the function%
\[%
\begin{array}
[c]{c}%
\Delta_{B,\omega_{0}}\varphi(t)=tr_{\omega_{0}}\omega(t)-tr_{\omega_{0}%
}\widehat{\omega}_{t},
\end{array}
\]
is uniformly bounded in $C^{k}(K,\omega_{0})$ for all $k\geq0.$ Form (\ref{6})
$\varphi(t)$ is uniformly bounded on $K$ and elliptic estimates
\[%
\begin{array}
[c]{c}%
||\varphi(t)||_{C^{k}(K,\omega_{0})}\leq C_{K,k},
\end{array}
\]
for all $t\geq0,$ $k\geq0,$ up to shrinking $K$ again. Now for $t\geq1$,
(\ref{12}) yields
\[%
\begin{array}
[c]{c}%
\varphi^{\prime}(t)\leq Cte^{-t},
\end{array}
\]
and so
\[%
\begin{array}
[c]{c}%
\frac{\partial}{\partial t}(\varphi(t)+Ce^{-t}(1+t))=\varphi^{\prime
}(t)-Cte^{-t}\leq0.
\end{array}
\]
Then the function $\varphi(t)+Ce^{-t}(1+t)$ is thus nonincreasing and
uniformly bounded from below on compact subsets of $M\setminus\mathrm{Null}%
(-c_{1}^{B}(M))$, and so as $t\rightarrow\infty$ the functions $\varphi(t) $
converge pointwise to a function $\varphi_{\infty}$ on $C_{\mathrm{loc}%
}^{\infty}(M\setminus\mathrm{Null}(-c_{1}^{B}(M))$, which is smooth on
$M\setminus\mathrm{Null}(-c_{1}^{B}(M))$. Also (\ref{15}) shows that
$\omega_{\infty}:=\chi+\sqrt{-1}\partial_{B}\overline{\partial}_{B}%
\varphi_{\infty}$ is a smooth transverse K\"{a}hler metric on $M\setminus
\mathrm{Null}(-c_{1}^{B}(M))$. The flow equation (\ref{3}) implies that
$\varphi^{\prime}(t)$ also converges smoothly to some limit function.
Moreover, since $\varphi(t)$ converge smoothly to $\varphi_{\infty}$ on
compact subsets of $M\setminus\mathrm{Null}(-c_{1}^{B}(M)),$ it follows that
given any $x\in M\setminus\mathrm{Null}(-c_{1}^{B}(M))$ there is a sequence
$t_{i}\rightarrow\infty$ such that $\varphi^{\prime}(x,t_{i})\rightarrow0 $.
But since $\varphi^{\prime}(t)$ converges smoothly on compact sets to some
limit function, it yields that $\varphi^{\prime}(t)\rightarrow0$ in
$C_{\mathrm{loc}}^{\infty}(M\setminus\mathrm{Null}(-c_{1}^{B}(M)))$. Taking
then the limit as $t\rightarrow\infty$ in (\ref{3}) we obtain%
\[%
\begin{array}
[c]{c}%
0=\log\frac{\omega_{\infty}^{n}\wedge\eta_{0}}{\Omega\wedge\eta_{0}}%
-\varphi_{\infty},
\end{array}
\]
on $M\setminus\mathrm{Null}(-c_{1}^{B}(M))$. Taking $\sqrt{-1}\partial
_{B}\overline{\partial}_{B}$ of this equation, we finally obtain
\[%
\begin{array}
[c]{c}%
\mathrm{Ric}_{\omega_{\infty}}^{T}=-\chi-\sqrt{-1}\partial_{B}\partial
_{B}\varphi_{\infty}=-\omega_{\infty}%
\end{array}
\]
as described.
\end{proof}

\section{The Gradient Estimate}

In this section we show the transverse gradient estimate and uniformly bounded
of the transverse scalar curvature under the normalized Sasaki-Ricci flow
(\ref{3}).

We first prove the following parabolic Schwarz lemma.

\begin{lemma}
Let $\omega=\omega(t)$ be the solution to the normalized Sasaki-Ricci flow
(\ref{3}). Then there exists $C>0$ such that on $M\times\lbrack0,\infty),$%
\begin{equation}%
\begin{array}
[c]{c}%
(\frac{\partial}{\partial t}-\Delta_{B})tr_{\omega}\widehat{\omega}_{\infty
}\leq tr_{\omega}\chi+C(tr_{\omega}\chi)^{2}-|\bigtriangledown^{T}tr_{\omega
}\chi|_{g^{T}}^{2},
\end{array}
\label{21}%
\end{equation}
where $\Delta_{B}$ is the basic Laplace operator associated to the evolving
transverse metric $g^{T}(t).$
\end{lemma}

\begin{proof}
Since the transverse canonical bundle $K_{M}^{T}$ is semi-ample. There exists
a basic transverse base point free holomorphic map $\Psi:M\rightarrow
(\mathbb{C}\mathbb{P}^{N},\omega_{FS})$ such that
\[%
\begin{array}
[c]{c}%
\widehat{\omega}_{\infty}=\frac{1}{m}\Psi^{\ast}(\omega_{FS}).
\end{array}
\]
Then we have%
\[%
\begin{array}
[c]{c}%
\frac{\partial}{\partial t}tr_{\omega}\widehat{\omega}_{\infty}=\left\langle
\mathrm{Ric}^{T},\widehat{\omega}_{\infty}\right\rangle _{\omega}+tr_{\omega
}\widehat{\omega}_{\infty}%
\end{array}
\]
and
\[%
\begin{array}
[c]{lll}%
\Delta_{B}tr_{\omega}\widehat{\omega}_{\infty} & = & \left\langle
\mathrm{Ric}^{T},\widehat{\omega}_{\infty}\right\rangle _{\omega
}+|\bigtriangledown^{T}tr_{\omega}\widehat{\omega}_{\infty}|_{g^{T}}%
^{2}-g^{Ti\overline{j}}g^{Tk\overline{l}}S_{\alpha\overline{\beta}%
\gamma\overline{\delta}}\psi_{i}^{\alpha}\psi_{\overline{j}}^{\overline{\beta
}}\psi_{k}^{\gamma}\psi_{\overline{l}}^{\overline{\delta}}\\
& \geq & \left\langle \mathrm{Ric}^{T},\widehat{\omega}_{\infty}\right\rangle
_{\omega}+|\bigtriangledown^{T}tr_{\omega}\widehat{\omega}_{\infty}|_{g^{T}%
}^{2}-C(tr_{\omega}\widehat{\omega}_{\infty})^{2},
\end{array}
\]
where $C$ is a universal constant given by the upper bound of the bisection
curvature $S_{\alpha\overline{\beta}\gamma\overline{\delta}}$ of $\omega_{FS}$
on $\mathbb{C}\mathbb{P}^{n}.$ This implies the inequality (\ref{21}).
\end{proof}

\begin{proposition}
\label{prop 1}There exists $C>0$ such that on $M\times\lbrack0,\infty),$
\[%
\begin{array}
[c]{c}%
tr_{\omega}\widehat{\omega}_{\infty}\leq C.
\end{array}
\]

\end{proposition}

\begin{proof}
By (\ref{21}), we compute%
\[%
\begin{array}
[c]{lll}%
(\frac{\partial}{\partial t}-\Delta_{B})\log(tr_{\omega}\widehat{\omega
}_{\infty}+1) & = & \frac{1}{tr_{\omega}\widehat{\omega}_{\infty}+1}%
(\frac{\partial}{\partial t}-\Delta^{T})tr_{\omega}\widehat{\omega}_{\infty
}+\frac{|\bigtriangledown^{T}tr_{\omega}\widehat{\omega}_{\infty}|_{g^{T}}%
^{2}}{(tr_{\omega}\widehat{\omega}_{\infty}+1)^{2}}\\
& \leq & 1+C(tr_{\omega}\chi).
\end{array}
\]
Then from the evolutions of $\varphi$ and $\varphi^{\prime},$ we obtain%
\[%
\begin{array}
[c]{c}%
(\frac{\partial}{\partial t}-\Delta_{B})[\log(tr_{\omega}\widehat{\omega
}_{\infty}+1)-A(\varphi+\varphi^{\prime})]\leq-(A-C)tr_{\omega}\widehat{\omega
}_{\infty}+An+1,
\end{array}
\]
for a large positive constant $A$ which is larger than $C$. By the maximum
principle, $tr_{\omega}\widehat{\omega}_{\infty}$ is uniformly bounded from
above on $M\times\lbrack0,\infty).$
\end{proof}

Denote
\[%
\begin{array}
[c]{c}%
u=\varphi+\varphi^{\prime}.
\end{array}
\]
Since both $\varphi$ and $\varphi^{\prime}$ are uniformly bounded from
(\ref{6}), there exists a positive constant $A$ such that
\[%
\begin{array}
[c]{c}%
u+A\geq1
\end{array}
\]
on $M\times\lbrack0,\infty).$

\begin{proposition}
There exists $C>0$ such that on $M\times\lbrack0,\infty),$ we have%
\begin{equation}%
\begin{array}
[c]{c}%
|\nabla^{T}u|_{g^{T}}^{2}\leq C
\end{array}
\label{20}%
\end{equation}
and%
\begin{equation}%
\begin{array}
[c]{c}%
|\Delta_{B}u|\leq C.
\end{array}
\label{20a}%
\end{equation}

\end{proposition}

\begin{proof}
We compute the evolution equations of $|\nabla^{T}u|_{g^{T}}^{2}$ and
$\Delta_{B}u$ as below. First note that
\[%
\begin{array}
[c]{c}%
(\frac{\partial}{\partial t}-\Delta_{B})u=tr_{\omega}\widehat{\omega}_{\infty
}-n.
\end{array}
\]
We obtain%
\[%
\begin{array}
[c]{lll}%
\frac{\partial}{\partial t}|\bigtriangledown^{T}u|^{2} & = & \left\langle
\bigtriangledown^{T}u,\bigtriangledown^{T}\Delta_{B}u\right\rangle
+\left\langle \bigtriangledown^{T}\Delta_{B}u,\bigtriangledown^{T}%
u\right\rangle +|\bigtriangledown^{T}u|^{2}\\
&  & +\mathrm{Ric}^{T}(\bigtriangledown^{T}u,\bigtriangledown^{T}%
u)+2\operatorname{Re}(\bigtriangledown^{T}tr_{\omega}\widehat{\omega}_{\infty
}\cdot\overline{\bigtriangledown}^{T}u).
\end{array}
\]
On the other hand, the Bonchner formula for the transverse Laplacian
$\Delta_{B}$ gives%
\[%
\begin{array}
[c]{lll}%
\Delta_{B}|\bigtriangledown^{T}u|^{2} & = & |\bigtriangledown^{T}%
\overline{\bigtriangledown}^{T}u|^{2}+|\bigtriangledown^{T}\bigtriangledown
^{T}u|^{2}+\left\langle \bigtriangledown^{T}u,\bigtriangledown^{T}\Delta
_{B}u\right\rangle \\
&  & +\left\langle \bigtriangledown^{T}\Delta_{B}u,\bigtriangledown
^{T}u\right\rangle +\mathrm{Ric}^{T}(\bigtriangledown^{T}u,\bigtriangledown
^{T}u).
\end{array}
\]
Hence%
\begin{equation}%
\begin{array}
[c]{lll}%
(\frac{\partial}{\partial t}-\Delta_{B})|\bigtriangledown^{T}u|^{2} & = &
|\bigtriangledown^{T}u|^{2}-|\bigtriangledown^{T}\overline{\bigtriangledown
}^{T}u|^{2}-|\bigtriangledown^{T}\bigtriangledown^{T}u|^{2}\\
&  & +2\operatorname{Re}(\bigtriangledown^{T}tr_{\omega}\widehat{\omega
}_{\infty}\cdot\overline{\bigtriangledown}^{T}u).
\end{array}
\label{22a}%
\end{equation}
Also
\begin{equation}%
\begin{array}
[c]{lll}%
(\frac{\partial}{\partial t}-\Delta_{B})\Delta_{B}u & = & \Delta
_{B}u+\left\langle \mathrm{Ric}^{T},\partial_{B}\overline{\partial}%
_{B}u\right\rangle +\Delta_{B}tr_{\omega}\widehat{\omega}_{\infty}\\
& = & -|\bigtriangledown^{T}\overline{\bigtriangledown}^{T}u|^{2}-\left\langle
\chi,\partial_{B}\overline{\partial}_{B}u\right\rangle +\Delta_{B}u+\Delta
_{B}tr_{\omega}\widehat{\omega}_{\infty}.
\end{array}
\label{22b}%
\end{equation}
Let%
\begin{equation}%
\begin{array}
[c]{c}%
H=\frac{|\bigtriangledown^{T}u|^{2}}{u+A}+tr_{\omega}\widehat{\omega}_{\infty
}.
\end{array}
\label{22}%
\end{equation}
Then%
\[%
\begin{array}
[c]{c}%
H_{t}=\frac{1}{u+A}\frac{\partial}{\partial t}|\bigtriangledown^{T}%
u|^{2}-\frac{|\bigtriangledown^{T}u|^{2}}{(u+A)^{2}}u_{t}+\frac{\partial
}{\partial t}tr_{\omega}\widehat{\omega}_{\infty}%
\end{array}
\]
and%
\begin{equation}%
\begin{array}
[c]{c}%
\bigtriangledown^{T}H=\frac{1}{u+A}\bigtriangledown^{T}|\bigtriangledown
^{T}u|^{2}-\frac{|\bigtriangledown^{T}u|^{2}}{(u+A)^{2}}\bigtriangledown
^{T}u+\bigtriangledown^{T}tr_{\omega}\widehat{\omega}_{\infty}.
\end{array}
\label{23}%
\end{equation}
On the other hand, since $|\bigtriangledown^{T}u|^{2}=(H-tr_{\omega
}\widehat{\omega}_{\infty})(u+A),$ we get%
\[%
\begin{array}
[c]{lll}%
\Delta_{B}|\bigtriangledown^{T}u|^{2} & = & (\Delta_{B}H-\Delta_{B}tr_{\omega
}\widehat{\omega}_{\infty})(u+A)+(H-tr_{\omega}\widehat{\omega}_{\infty
})\Delta_{B}u\\
&  & +2\operatorname{Re}(\bigtriangledown^{T}(H-tr_{\omega}\chi)\cdot
\overline{\bigtriangledown}^{T}u).
\end{array}
\]
or
\[%
\begin{array}
[c]{lll}%
\Delta_{B}H & = & \frac{1}{u+A}[\Delta_{B}|\bigtriangledown^{T}u|^{2}%
-\frac{|\bigtriangledown^{T}u|^{2}}{u+A}\Delta_{B}u]+\Delta_{B}tr_{\omega
}\widehat{\omega}_{\infty}\\
&  & -\frac{2}{u+A}\operatorname{Re}[\bigtriangledown^{T}(H-tr_{\omega
}\widehat{\omega}_{\infty})\cdot\overline{\bigtriangledown}^{T}u].
\end{array}
\]
Thus%
\[%
\begin{array}
[c]{lll}%
(\frac{\partial}{\partial t}-\Delta_{B})H & = & \frac{1}{u+A}(\frac{\partial
}{\partial t}-\Delta_{B})|\bigtriangledown^{T}u|^{2}-\frac{|\bigtriangledown
^{T}u|^{2}}{(u+A)^{2}}(\frac{\partial}{\partial t}-\Delta_{B})u+(\frac
{\partial}{\partial t}-\Delta_{B})tr_{\omega}\widehat{\omega}_{\infty}\\
&  & +\frac{2}{u+A}\operatorname{Re}[\bigtriangledown^{T}(H-tr_{\omega
}\widehat{\omega}_{\infty})\cdot\overline{\bigtriangledown}^{T}u].
\end{array}
\]
Using (\ref{23}) and express%
\[%
\begin{array}
[c]{c}%
\frac{\bigtriangledown^{T}H\cdot\overline{\bigtriangledown}^{T}u}%
{u+A}=(1-2\epsilon)\frac{\bigtriangledown^{T}H\cdot\overline{\bigtriangledown
}^{T}u}{u+A}+\frac{2\epsilon}{u+A}[(\frac{\bigtriangledown^{T}%
|\bigtriangledown^{T}u|^{2}}{u+A}+\bigtriangledown^{T}tr_{\omega
}\widehat{\omega}_{\infty})\cdot\overline{\bigtriangledown}^{T}u-\frac
{|\bigtriangledown^{T}u|^{4}}{(u+A)^{2}}].
\end{array}
\]
Therefore, from
\[%
\begin{array}
[c]{lll}%
(\frac{\partial}{\partial t}-\Delta_{B})H & = & \frac{1}{u+A}%
[|\bigtriangledown^{T}u|^{2}-|\bigtriangledown^{T}\overline{\bigtriangledown
}^{T}u|^{2}-|\bigtriangledown^{T}\bigtriangledown^{T}u|^{2}]+(\frac{\partial
}{\partial t}-\Delta_{B})tr_{\omega}\widehat{\omega}_{\infty}\\
&  & +\frac{4\epsilon}{u+A}\operatorname{Re}[(\frac{1}{u+A}\bigtriangledown
^{T}|\bigtriangledown^{T}u|^{2}+\bigtriangledown^{T}tr_{\omega}\widehat{\omega
}_{\infty}\cdot\overline{\bigtriangledown}^{T}u]-4\epsilon\frac
{|\bigtriangledown^{T}u|^{4}}{(u+A)^{3}}\\
&  & +\frac{2(1-2\epsilon)}{u+A}\operatorname{Re}(\bigtriangledown^{T}%
H\cdot\overline{\bigtriangledown}^{T}u)+\frac{|\bigtriangledown^{T}u|^{2}%
}{(u+A)^{2}}(n-tr_{\omega}\widehat{\omega}_{\infty})\\
& \leq & \frac{|\bigtriangledown^{T}u|^{2}}{u+A}-\frac{1-4\epsilon}%
{u+A}[|\bigtriangledown^{T}\overline{\bigtriangledown}^{T}u|^{2}%
+|\bigtriangledown^{T}\bigtriangledown^{T}u|^{2}]-2\epsilon\frac
{|\bigtriangledown^{T}u|^{4}}{(u+A)^{3}}\\
&  & -\frac{\epsilon}{u+A}[|2\bigtriangledown^{T}\overline{\bigtriangledown
}^{T}u-\frac{\bigtriangledown^{T}u\overline{\bigtriangledown}^{T}u}{u+A}%
|^{2}+|2\bigtriangledown^{T}\bigtriangledown^{T}u-\frac{\bigtriangledown
^{T}u\bigtriangledown^{T}u}{u+A}|^{2}]\\
&  & +\frac{2(1-2\epsilon)}{u+A}\operatorname{Re}(\bigtriangledown^{T}%
H\cdot\overline{\bigtriangledown}^{T}u)+\frac{|\bigtriangledown^{T}u|^{2}%
}{(u+A)^{2}}(n-tr_{\omega}\widehat{\omega}_{\infty})\\
&  & +C(tr_{\omega}\widehat{\omega}_{\infty})^{2}+tr_{\omega}\widehat{\omega
}_{\infty}-|\bigtriangledown^{T}tr_{\omega}\widehat{\omega}_{\infty}%
|^{2}+\frac{8\epsilon}{u+A}|\bigtriangledown^{T}tr_{\omega}\widehat{\omega
}_{\infty}||\bigtriangledown^{T}u|\\
& \leq & -C\epsilon|\bigtriangledown^{T}u|^{4}+\frac{2(1-2\epsilon)}%
{u+A}\operatorname{Re}(\bigtriangledown^{T}H\cdot\overline{\bigtriangledown
}^{T}u)+C.
\end{array}
\]
For any $T>0,$ suppose $H$ attains its maximum at $(x_{0},t_{0})$ on
$M\times\lbrack0,T],$ then%
\begin{equation}%
\begin{array}
[c]{c}%
H_{t}(x_{0},t_{0})\geq0,\text{ \ }\bigtriangledown^{T}H(x_{0},t_{0})=0,\text{
\ \textrm{and} \ }\Delta_{B}H(x_{0},t_{0})\leq0.
\end{array}
\label{24}%
\end{equation}
Thus, by choosing $\epsilon=1/8,$ we get that $|\bigtriangledown^{T}%
u|^{4}(x_{0},t_{0})\leq C$ and by the uniform bound of $tr_{\omega}\chi,$%
\[%
\begin{array}
[c]{c}%
H(x_{0},t_{0})\leq C.
\end{array}
\]
Therefore, since $T>0$ is arbitrary, we then arrive at
\[%
\begin{array}
[c]{c}%
\frac{|\bigtriangledown^{T}u|^{2}}{u+A}+tr_{\omega}\widehat{\omega}_{\infty
}\leq C
\end{array}
\]
on $M\times\lbrack0,\infty).$Now\ we show the inequality (\ref{20a}). Let
\[%
\begin{array}
[c]{c}%
K=2\frac{|\bigtriangledown^{T}u|^{2}}{u+A}-\frac{\Delta^{T}u}{u+A},
\end{array}
\]
using (\ref{22b}) and (\ref{22a}), the evolution equation for $K$ is given by
\[%
\begin{array}
[c]{ll}
& (\frac{\partial}{\partial t}-\Delta_{B})K\\
= & \frac{1}{u+A}\{(\frac{\partial}{\partial t}-\Delta_{B})[2|\bigtriangledown
^{T}u|^{2}-\Delta_{B}u]+K(n-tr_{\omega}\widehat{\omega}_{\infty}%
)+2\operatorname{Re}(\bigtriangledown^{T}K\cdot\overline{\bigtriangledown}%
^{T}u)\}\\
= & \frac{1}{u+A}[2|\bigtriangledown^{T}u|^{2}-\Delta_{B}u-2|\bigtriangledown
^{T}\bigtriangledown^{T}u|^{2}-|\bigtriangledown^{T}\overline{\bigtriangledown
}^{T}u|^{2}]\\
& +\frac{1}{u+A}[\left\langle \widehat{\omega}_{\infty},\partial_{B}%
\overline{\partial}_{B}u\right\rangle -\Delta_{B}tr_{\omega}\widehat{\omega
}_{\infty}+K(n-tr_{\omega}\widehat{\omega}_{\infty})]\\
& +\frac{2}{u+A}\operatorname{Re}[\bigtriangledown^{T}(K+2tr_{\omega
}\widehat{\omega}_{\infty})\cdot\overline{\bigtriangledown}^{T}u].
\end{array}
\]
From (\ref{21}), Proposition \ref{prop 1} and $\mathrm{Ric}^{T}=-\partial
_{B}\overline{\partial}_{B}u-\widehat{\omega}_{\infty}$, we estimate the term
$-\Delta_{B}tr_{\omega}\widehat{\omega}_{\infty}$ as follows%
\[%
\begin{array}
[c]{lll}%
-\Delta_{B}tr_{\omega}\widehat{\omega}_{\infty} & = & (\frac{\partial
}{\partial t}-\Delta_{B})tr_{\omega}\widehat{\omega}_{\infty}-\frac{\partial
}{\partial t}tr_{\omega}\widehat{\omega}_{\infty}\\
& \leq & C(tr_{\omega}\widehat{\omega}_{\infty})^{2}-|\bigtriangledown
^{T}tr_{\omega}\widehat{\omega}_{\infty}|^{2}+\left\langle \partial
_{B}\overline{\partial}_{B}u+\widehat{\omega}_{\infty},\widehat{\omega
}_{\infty}\right\rangle \\
& \leq & \frac{1}{4}|\bigtriangledown^{T}\overline{\bigtriangledown}^{T}%
u|^{2}-|\bigtriangledown^{T}tr_{\omega}\widehat{\omega}_{\infty}|^{2}+C.
\end{array}
\]
By combining the above estimate with inequalities (\ref{20}) and
$|\bigtriangledown^{T}\overline{\bigtriangledown}^{T}u|^{2}\geq(\Delta
_{B}u)^{2}/n$ and applying Schwarz inequality
\[%
\begin{array}
[c]{c}%
(\frac{\partial}{\partial t}-\Delta_{B})K\leq-\frac{1}{4n(u+A)}(\Delta
_{B}u)^{2}+\frac{2}{u+A}\operatorname{Re}(\bigtriangledown^{T}K\cdot
\overline{\bigtriangledown}^{T}u)+C.
\end{array}
\]
Again for any $T>0,$ suppose $K$ attains its maximum at $(x_{0},t_{0})$ on
$M\times\lbrack0,T],$ then the conditions (\ref{24}) holds for $K$, and hence
$(\Delta_{B}u)(x_{0},t_{0})$ is bounded uniformly. Therefore, by
(\ref{20})\ again, $(\Delta_{B}u)(x,t)$ is bounded uniformly on $M\times
\lbrack0,T]$ for arbitrary $T>0$.
\end{proof}

The transverse scalar curvature $R^{T}(t)$ along the normalized Sasaki-Ricci
flow (\ref{3}) is expressed by%
\begin{equation}%
\begin{array}
[c]{c}%
R^{T}(t)=-\Delta_{B}u-tr_{\omega}\widehat{\omega}_{\infty}.
\end{array}
\label{25}%
\end{equation}
Recall the evolution of the transverse scalar curvature $R^{T},$%
\[%
\begin{array}
[c]{c}%
(\frac{\partial}{\partial t}-\Delta_{B})R^{T}=|\mathrm{Ric}^{T}|^{2}+R^{T}.
\end{array}
\]
By the maximum principle, $R^{T}(t)$ is uniformly bounded from below on
$M\times\lbrack0,\infty)$ and it is also uniformly bounded from above by
(\ref{20a}) and Proposition \ref{prop 1}.

\begin{proposition}
There exists $C>0$ such that on $M\times\lbrack0,\infty),$ we have%
\begin{equation}%
\begin{array}
[c]{c}%
|R^{T}(t)|\leq C.
\end{array}
\label{aaa}%
\end{equation}

\end{proposition}

\section{$L^{4}$-Bound of the Transverse Ricci Curvature}

In this section, we show the $L^{4}$-bound of the transverse Ricci curvature
under the normalized Sasaki-Ricci flow (\ref{3}).

\begin{theorem}
\label{T51}Let $(M,\eta_{0},\xi_{0},\Phi_{0},g_{0},\omega_{0})$ be a compact
quasi-regular Sasakian $(2n+1)$-manifold and its the space $Z$ of leaves of
the characteristic foliation be well-formed. Suppose that $K_{M}^{T}$ is nef
and big. Then there exists a positive constant $C$ such that
\begin{equation}%
\begin{array}
[c]{c}%
\int_{t}^{t+1}\int_{M}|\mathrm{Ric}_{\omega(s)}^{T}|^{4}\omega(s)^{n}%
\wedge\eta_{0}ds\leq C,
\end{array}
\label{41}%
\end{equation}
for all $t\in\lbrack0,\infty).$ Moreover, for any $0<p<4,$ we have%
\begin{equation}%
\begin{array}
[c]{c}%
\int_{t}^{t+1}\int_{M}|\mathrm{Ric}_{\omega(s)}^{T}+\omega(s)|^{p}%
\omega(s)^{n}\wedge\eta_{0}ds\rightarrow0\text{ \textrm{as} }t\rightarrow
\infty.
\end{array}
\label{40}%
\end{equation}

\end{theorem}

Since the transverse canonical bundle $K_{M}^{T}$ is semi-ample. There exists
a basic transverse holomorphic map $\Psi:M\rightarrow(\mathbb{C}\mathrm{P}%
^{N},\omega_{FS})$ such that $\widehat{\omega}_{\infty}=\frac{1}{m}\Psi^{\ast
}(\omega_{FS})$ and
\begin{equation}%
\begin{array}
[c]{c}%
\mathrm{Ric}_{\omega(t)}^{T}+\sqrt{-1}\partial_{B}\overline{\partial}%
_{B}u(t)=-\widehat{\omega}_{\infty}.
\end{array}
\label{41a}%
\end{equation}
So, by the uniform bound of $\widehat{\omega}_{\infty}$ in terms of
$\omega(t),$ to prove the $L^{4}$ bound of transverse Ricci curvature
(\ref{41}) under the normalized Sasaki-Ricci flow, it suffices to show that%
\begin{equation}%
\begin{array}
[c]{c}%
\int_{t}^{t+1}\int_{M}|\partial_{B}\overline{\partial}_{B}u(s)|^{4}%
\omega(s)^{n}\wedge\eta_{0}ds\leq C,
\end{array}
\label{42}%
\end{equation}
for all $t\geq0$ and for some constant $C$ independent of $t$. We need the
following Lemmas.

\begin{lemma}
\label{L4.1}There exists a positive constant $C=C(\omega_{0},\widehat{\omega
}_{\infty})$ such that%
\begin{equation}%
\begin{array}
[c]{c}%
\int_{M}[|\bigtriangledown^{T}\overline{\bigtriangledown}^{T}u|^{2}%
+|\bigtriangledown^{T}\bigtriangledown^{T}u|^{2}+|\mathrm{Ric}^{T}%
|^{2}+|\mathrm{Rm}^{T}|^{2}]\omega(t)^{n}\wedge\eta_{0}\leq C,
\end{array}
\label{43}%
\end{equation}
for all $t\in\lbrack0,\infty).$
\end{lemma}

\begin{proof}
From the relation (\ref{41})
\[%
\begin{array}
[c]{c}%
\int_{M}|\mathrm{Ric}^{T}|^{2}\omega(t)^{n}\wedge\eta_{0}\leq\int%
_{M}[|\bigtriangledown^{T}\overline{\bigtriangledown}^{T}u|^{2}+(tr_{\omega
}\widehat{\omega}_{\infty})^{2}]\omega(t)^{n}\wedge\eta_{0}.
\end{array}
\]
Applying the integration by parts, we have
\[%
\begin{array}
[c]{c}%
\int_{M}[|\bigtriangledown^{T}\overline{\bigtriangledown}^{T}u|^{2}%
\omega(t)^{n}\wedge\eta_{0}=\int_{M}(\Delta_{B}u)^{2}\omega(t)^{n}\wedge
\eta_{0}%
\end{array}
\]
and also%
\[%
\begin{array}
[c]{ll}
& \int_{M}|\bigtriangledown^{T}\bigtriangledown^{T}u|^{2}\omega(t)^{n}%
\wedge\eta_{0}\\
= & \int_{M}[(\Delta_{B}u)^{2}-\left\langle \mathrm{Ric}^{T},\partial
_{B}u\overline{\partial}_{B}u\right\rangle ]\omega(t)^{n}\wedge\eta_{0}\\
\leq & \int_{M}[(\Delta_{B}u)^{2}+|\bigtriangledown^{T}\overline
{\bigtriangledown}^{T}u|^{2}+(tr_{\omega}\widehat{\omega}_{\infty}%
)^{2}+|\bigtriangledown^{T}u|^{4}]\omega(t)^{n}\wedge\eta_{0}.
\end{array}
\]
Moreover, the $L^{2}$-bound of the transverse Riemannian curvature tensor
follows from (\ref{aaa}) and the Sasaki analogue of the Chern-Weil theory as
in \cite[Lemma 7.2]{zh}:%
\[%
\begin{array}
[c]{l}%
\begin{array}
[c]{c}%
\int_{M}(2\pi)^{2}[2c_{2}^{B}-\frac{n}{n+1}(c_{1}^{B})^{2}]\wedge\frac
{1}{2^{n-2}(n-2)!}\omega(t)^{n-2}\wedge\eta_{0}%
\end{array}
\\
=\int_{M}[|Rm^{T}|-\frac{2}{n(n+1)}(R^{T})^{2}-\frac{(n-1)(n+2)}%
{n(n+1)}((R^{T})^{2}+(2n(n+1))^{2})]\frac{1}{2^{n}n!}\omega(t)^{n}\wedge
\eta_{0}.
\end{array}
\]

\end{proof}

The following integral inequalities hold for any smooth basic function on $M$.

\begin{lemma}
\label{L4.2}There exists a universal positive constant $C=C(n)$ such that%
\begin{equation}%
\begin{array}
[c]{l}%
\int_{M}|\bigtriangledown^{T}\overline{\bigtriangledown}^{T}u|^{4}\omega
^{n}\wedge\eta_{0}\\
\leq C\int_{M}|\bigtriangledown^{T}u|^{2}[|\overline{\bigtriangledown}%
^{T}\bigtriangledown^{T}\bigtriangledown^{T}u|^{2}+|\bigtriangledown
^{T}\bigtriangledown^{T}\overline{\bigtriangledown}^{T}u|^{2}]\omega^{n}%
\wedge\eta_{0}%
\end{array}
\label{44}%
\end{equation}
and%
\begin{equation}%
\begin{array}
[c]{c}%
\int_{M}[|\overline{\bigtriangledown}^{T}\bigtriangledown^{T}\bigtriangledown
^{T}u|^{2}+|\bigtriangledown^{T}\bigtriangledown^{T}\overline{\bigtriangledown
}^{T}u|^{2}]\omega^{n}\wedge\eta_{0}\\
\leq C\int_{M}[|\bigtriangledown^{T}\Delta^{T}u|^{2}+|\bigtriangledown
^{T}u|^{2}|\mathrm{Rm}^{T}|^{2}]\omega^{n}\wedge\eta_{0}.
\end{array}
\label{45}%
\end{equation}

\end{lemma}

\begin{proposition}
There exists a positive constant $C=C(\omega_{0},\widehat{\omega}_{\infty})$
such that%
\begin{equation}%
\begin{array}
[c]{c}%
\int_{t}^{t+1}\int_{M}[|\bigtriangledown^{T}\overline{\bigtriangledown}%
^{T}u|^{4}+|\overline{\bigtriangledown}^{T}\bigtriangledown^{T}%
\bigtriangledown^{T}u|^{2}+|\bigtriangledown^{T}\bigtriangledown^{T}%
\overline{\bigtriangledown}^{T}u|^{2}]\omega(s)^{n}\wedge\eta_{0}ds\leq C,
\end{array}
\label{46}%
\end{equation}
for all $t\in\lbrack0,\infty).$
\end{proposition}

\begin{proof}
By the previous Lemmas \ref{L4.1} and \ref{L4.2}, it sufficient to prove a
uniform $L^{2}$ bound $\bigtriangledown^{T}\Delta_{B}u.$ Since
\[%
\begin{array}
[c]{c}%
(\frac{\partial}{\partial t}-\Delta_{B})\Delta_{B}u=\Delta_{B}u+\left\langle
\mathrm{Ric}^{T},\partial_{B}\overline{\partial}_{B}u\right\rangle _{\omega
}+\Delta_{B}tr_{\omega}\widehat{\omega}_{\infty},
\end{array}
\]
thus%
\[%
\begin{array}
[c]{c}%
\frac{1}{2}(\frac{\partial}{\partial t}-\Delta_{B})(\Delta_{B}u)^{2}%
=(\Delta_{B}u)^{2}-|\bigtriangledown^{T}\Delta_{B}u|^{2}+\Delta_{B}%
u[\left\langle \mathrm{Ric}^{T},\partial_{B}\overline{\partial}_{B}%
u\right\rangle _{\omega}+\Delta_{B}tr_{\omega}\widehat{\omega}_{\infty}].
\end{array}
\]
Integrating over the manifold gives%
\[%
\begin{array}
[c]{ll}
& \int_{M}|\bigtriangledown^{T}\Delta_{B}u|^{2}\omega^{n}\wedge\eta_{0}\\
\leq & \int_{M}[(\Delta_{B}u)^{2}+|\Delta_{B}u||\mathrm{Ric}^{T}%
||\bigtriangledown^{T}\overline{\bigtriangledown}^{T}u|-2\operatorname{Re}%
(\bigtriangledown^{T}\Delta_{B}u\cdot\overline{\bigtriangledown}^{T}%
tr_{\omega}\widehat{\omega}_{\infty})]\omega^{n}\wedge\eta_{0}\\
& -\frac{1}{2}\int_{M}\frac{\partial}{\partial t}(\Delta_{B}u)^{2}\omega
^{n}\wedge\eta_{0}\\
\leq & \int_{M}[\frac{1}{2}|\bigtriangledown^{T}\Delta_{B}u|^{2}+(\Delta
_{B}u)^{2}(1+|\mathrm{Ric}^{T}|^{2})+|\bigtriangledown^{T}\overline
{\bigtriangledown}^{T}u|^{2}]\omega^{n}\wedge\eta_{0}\\
& +\int_{M}[2|\bigtriangledown^{T}tr_{\omega}\widehat{\omega}_{\infty}%
|^{2}-\frac{1}{2}(\Delta_{B}u)^{2}(R^{T}+n)]\omega^{n}\wedge\eta_{0}-\frac
{1}{2}\frac{d}{dt}\int_{M}(\Delta_{B}u)^{2}\omega^{n}\wedge\eta_{0}.
\end{array}
\]
Applying the uniform bound of $\Delta_{B}u$ and Lemma \ref{L4.1}, we then
obtain%
\[%
\begin{array}
[c]{lll}%
\int_{M}|\bigtriangledown^{T}\Delta_{B}u|^{2}\omega^{n}\wedge\eta_{0} & \leq &
C\int_{M}[1+|\bigtriangledown^{T}tr_{\omega}\widehat{\omega}_{\infty}%
|^{2}]\omega^{n}\wedge\eta_{0}\\
&  & -\frac{d}{dt}\int_{M}(\Delta_{B}u)^{2}\omega^{n}\wedge\eta_{0}.
\end{array}
\]
Integrating over the time interval $[t,t+1],$ we have%
\begin{equation}%
\begin{array}
[c]{l}%
\int_{t}^{t+1}\int_{M}|\bigtriangledown^{T}\Delta_{B}u|^{2}\omega(s)^{n}%
\wedge\eta_{0}ds\\
\leq C\int_{t}^{t+1}\int_{M}(1+|\bigtriangledown^{T}tr_{\omega}\widehat{\omega
}_{\infty}|^{2})\omega(s)^{n}\wedge\eta_{0}ds+C,
\end{array}
\label{47}%
\end{equation}
for all $t\geq0.$ The integral of the term $|\bigtriangledown^{T}tr_{\omega
}\widehat{\omega}_{\infty}|$ can be estimated by the Schwarz lemma. From the
evolution equation of $tr_{\omega}\widehat{\omega}_{\infty}$%
\[%
\begin{array}
[c]{lll}%
-\Delta_{B}tr_{\omega}\widehat{\omega}_{\infty} & = & (\frac{\partial
}{\partial t}-\Delta_{B})tr_{\omega}\widehat{\omega}_{\infty}-\frac{\partial
}{\partial t}tr_{\omega}\widehat{\omega}_{\infty}\\
& \leq & C(tr_{\omega}\widehat{\omega}_{\infty})^{2}-\left\langle
\mathrm{Ric}^{T},\widehat{\omega}_{\infty}\right\rangle _{\omega
}-|\bigtriangledown^{T}tr_{\omega}\widehat{\omega}_{\infty}|\\
& \leq & C(tr_{\omega}\widehat{\omega}_{\infty})^{2}+|\mathrm{Ric}%
^{T}||tr_{\omega}\widehat{\omega}_{\infty}|-|\bigtriangledown^{T}tr_{\omega
}\widehat{\omega}_{\infty}|,
\end{array}
\]
where $C$ is a universal constant given by the upper bound of the bisection
curvature of $\omega_{FS}$ on $\mathbb{C}\mathbb{P}^{n}.$ Because
$0<tr_{\omega}\widehat{\omega}_{\infty}\leq C$ under the flow, we have%
\[%
\begin{array}
[c]{l}%
|\bigtriangledown^{T}tr_{\omega}\widehat{\omega}_{\infty}|\leq\Delta
_{B}tr_{\omega}\widehat{\omega}_{\infty}+C(|\mathrm{Ric}^{T}|+1)
\end{array}
\]
and thus%
\[%
\begin{array}
[c]{c}%
\int_{M}|\bigtriangledown^{T}tr_{\omega}\widehat{\omega}_{\infty}|^{2}%
\omega(t)^{n}\wedge\eta_{0}\leq C\int_{M}(|\mathrm{Ric}^{T}|+1)\omega
(t)^{n}\wedge\eta_{0}\leq C
\end{array}
\]
uniformly. Substituting into (\ref{47}) we obtain the desired estimate.
\end{proof}

In order to prove (\ref{40}) we use the $L^{2}$ estimate to the traceless
transverse Ricci curvature as following.

\begin{lemma}
Under the Sasaki-Ricci flow,%
\begin{equation}%
\begin{array}
[c]{c}%
\int_{t}^{t+1}\int_{M}|\mathrm{Ric}_{\omega(s)}^{T}+\omega(s)|^{2}%
\omega(s)^{n}\wedge\eta_{0}ds\rightarrow0\text{ \textrm{as} }t\rightarrow
\infty.
\end{array}
\label{48}%
\end{equation}

\end{lemma}

\begin{proof}
Recall the evolution of the transverse scalar curvature $R^{T}=tr_{\omega
}\mathrm{Ric}^{T}$%
\[%
\begin{array}
[c]{c}%
(\frac{\partial}{\partial t}-\Delta_{B})R^{T}=|\mathrm{Ric}^{T}|^{2}%
+R^{T}=|\mathrm{Ric}^{T}+\omega|^{2}-(R^{T}+n).
\end{array}
\]
The maximum principle shows that $\frac{d}{dt}\inf R^{T}\geq-(\inf R^{T}+n), $
which implies%
\begin{equation}%
\begin{array}
[c]{c}%
\inf R^{T}+n\geq e^{-t}\min(\inf R^{T}(0)+n,0)\geq-Ce^{-t}%
\end{array}
\label{49}%
\end{equation}
for some positive constant $C=C(\omega_{0}).$ Then
\[%
\begin{array}
[c]{ll}
& \int_{M}|\mathrm{Ric}_{\omega}^{T}+\omega|^{2}\omega^{n}\wedge\eta_{0}\\
= & \int_{M}(\frac{\partial}{\partial t}R^{T}+R^{T}+n)\omega^{n}\wedge\eta
_{0}\\
= & \frac{d}{dt}\int_{M}R^{T}\omega^{n}\wedge\eta_{0}+\int_{M}(R^{T}%
+n)(R^{T}+1)\omega^{n}\wedge\eta_{0}\\
= & \frac{d}{dt}\int_{M}R^{T}\omega^{n}\wedge\eta_{0}+\int_{M}(R^{T}%
+n+Ce^{-t})(R^{T}+1)\omega^{n}\wedge\eta_{0}\\
& -Ce^{-t}\int_{M}(R^{T}+1)\omega^{n}\wedge\eta_{0}\\
\leq & \frac{d}{dt}\int_{M}R^{T}\omega^{n}\wedge\eta_{0}+C\int_{M}%
(R^{T}+n)\omega^{n}\wedge\eta_{0}+Ce^{-t}%
\end{array}
\]
where we used the uniform bound of transverse scalar curvature and volume form
$\omega(t)^{n}\wedge\eta_{0}.$ The integration of $R^{T}+n$ becomes%
\[%
\begin{array}
[c]{lll}%
\int_{M}(R^{T}+n)\omega^{n}\wedge\eta_{0} & = & n\int_{M}(\mathrm{Ric}%
^{T}+\omega)\wedge\omega^{n-1}\wedge\eta_{0}=n\int_{M}(-\widehat{\omega
}_{\infty}+\widehat{\omega})\wedge\widehat{\omega}^{n-1}\wedge\eta_{0}\\
& = & ne^{-t}\int_{M}(\omega_{0}-\widehat{\omega}_{\infty})\wedge
\widehat{\omega}^{n-1}\wedge\eta_{0}\leq Ce^{-t}.
\end{array}
\]
Then
\[%
\begin{array}
[c]{lll}%
\int_{0}^{\infty}\int_{M}|\mathrm{Ric}_{\omega}^{T}+\omega|^{2}\omega
(t)^{n}\wedge\eta_{0}dt & \leq & \underset{t\rightarrow\infty}{\lim}\int%
_{M}R^{T}(t)\omega(t)^{n}\wedge\eta_{0}-\int_{M}R^{T}(0)\omega_{0}^{n}%
\wedge\eta_{0}+C\\
& \leq & C.
\end{array}
\]
This estimate implies the lemma.
\end{proof}

The estimate (\ref{40}) when $2\leq p<4$ then is a direct consequence of the
H\"{o}lder inequality
\[%
\begin{array}
[c]{c}%
\int_{t}^{t+1}\int_{M}|\mathrm{Ric}_{\omega}^{T}+\omega|^{p}\leq\left(
\int_{t}^{t+1}\int_{M}|\mathrm{Ric}_{\omega}^{T}+\omega|^{4}\right)
^{\frac{p-2}{2}}\left(  \int_{t}^{t+1}\int_{M}|\mathrm{Ric}_{\omega}%
^{T}+\omega|^{2}\right)  ^{\frac{4-p}{2}}.
\end{array}
\]
When $0<p<2$ the estimate (\ref{40}) is obvious.

\section{Cheeger-Gromov Convergence}

Let $(M,\eta,\xi,\Phi,g)$ be a compact quasi-regular Sasakian $(2n+1)$%
-manifold and its leave space $Z$ of the characteristic foliation be
well-formed which means its orbifold singular locus and algebro-geometric
singular locus coincide. In this section if we assume that $K_{M}^{T}$ is nef
and big, we will show that the solution of Sasaki-Ricci flow (\ref{3})
converge in the Gromov-Hausdorff topology to an $\eta$-Einstein metric on the
transverse canonical model without any curvature assumption in the case of the
dimension less than or equal to $7$.

Note that from the definition of quasi-regular Sasakian manifolds, there is a
natural projection
\[
\Pi:(C(M),\overline{g},J,\overline{\omega})\rightarrow(Z,h,\omega_{h})
\]
satisfying the orbifold Riemannian submersion $\pi:(M,g,\omega)\rightarrow
(Z,h,\omega_{h})$ with $\omega=\pi^{\ast}(\omega_{h})$ such that
\[
\Pi|_{(M,g,\omega)}=\pi
\]
and the volume form of the K\"{a}hler cone metric on the cone $C(M)$
\begin{equation}
\overline{\omega}^{n+1}=r^{2n+1}(\Pi^{\ast}\omega_{h})^{n}\wedge
dr\wedge\overline{\eta},\label{2022-d}%
\end{equation}
and the volume form of the Sasaki metric on $M$%
\begin{equation}%
\begin{array}
[c]{c}%
i_{\frac{\partial}{\partial r}}\overline{\omega}^{n+1}=(\Pi^{\ast}\omega
_{h})^{n}\wedge\eta=\omega^{n}\wedge\eta.
\end{array}
\label{2022-c}%
\end{equation}
Furthermore, by adapting notions from Definition \ref{D2021} and \cite{cz},
$G_{_{i}}$ is the local uniformizing finite group acting on a smooth complex
space $\widetilde{U_{i}}$ such that the local uniformizing group injects into
$U(1)$ and the map%
\[%
\begin{array}
[c]{c}%
\varphi_{i}:U(1)\times\widetilde{U_{i}}\rightarrow U_{i}%
\end{array}
\]
is exactly $|G_{_{i}}|$-to-one on the complement of the orbifold locus. Then
we can have the following computation
\begin{equation}%
\begin{array}
[c]{ccl}%
\int_{Z}|\mathrm{Ric}_{\omega_{h}(t)}|^{p}\omega_{h}(t)^{n} & = & \sum
_{i}\frac{1}{|G_{_{i}}|}\int_{\widetilde{U_{i}}}\varphi_{i}|\mathrm{Ric}%
_{\omega_{h}(t)}|^{p}\omega_{h}(t)^{n}\\
& = & \sum_{i}\frac{1}{|G_{_{i}}|}\int_{U(1)\times\widetilde{U_{i}}}\pi^{\ast
}\varphi_{i}|\mathrm{Ric}_{\omega(t)}^{T}|^{p}\pi^{\ast}\omega_{h}%
(t)^{n}\wedge\eta\\
& = & \sum_{i}\int_{U_{i}}\pi^{\ast}\varphi_{i}|\mathrm{Ric}_{\omega(t)}%
^{T}|^{p}\pi^{\ast}\omega_{h}(t)^{n}\wedge\eta\\
& = & \int_{M}|\mathrm{Ric}_{\omega(t)}^{T}|^{p}i_{\frac{\partial}{\partial
r}}\overline{\omega}^{n+1}\\
& = & \int_{M}|\mathrm{Ric}_{\omega(t)}^{T}|^{p}\omega(t)^{n}\wedge\eta.
\end{array}
\label{2022-e}%
\end{equation}

With (\ref{2022-e}) in mind, we will apply our previous results plus
Cheeger-Colding-Tian structure theory for K\"{a}hler orbifolds (\cite{cct},
\cite{tz1} and \cite[Theorem2.3]{tz2}) to study the structure of desired limit
space. Since $(M,\eta,\xi,\Phi,g)$ is a compact quasi-regular Sasakian
manifold, by the%
 first structure theorem on Sasakian manifolds, %
$M$\ is a principal $S^{1}$-orbibundle ($V$-bundle) over $Z$ which is also a
$Q$-factorial, polarized, normal projective orbifold such that there is an
orbifold Riemannian submersion$\ \pi:(M,g,\omega)\rightarrow(Z,h,\omega_{h})$
with%
\[%
\begin{array}
[c]{c}%
g=g^{T}+\eta\otimes\eta
\end{array}
\]
and
\[%
\begin{array}
[c]{c}%
g^{T}=\pi^{\ast}(h),\text{\ }\frac{1}{2}d\eta=\pi^{\ast}(\omega_{h}).
\end{array}
\]
The orbit $\xi_{x}$ is compact for any $x\in M,$ we then define the transverse
distance function as
\[%
\begin{array}
[c]{c}%
d^{T}(x,y)\triangleq d_{g}(\xi_{x},\xi_{y}),
\end{array}
\]
where $d$ is the distance function defined by the Sasaki metric $g.$ Then
\[%
\begin{array}
[c]{c}%
d^{T}(x,y)=d_{h}(\pi(x),\pi(y)).
\end{array}
\]
We define a transverse ball $B_{\xi,g}(x,r)$ as follows:
\[%
\begin{array}
[c]{c}%
B_{\xi,g}(x,r)=\left\{  y:d^{T}(x,y)<r\right\}  =\left\{  y:d_{h}(\pi
(x),\pi(y))<r\right\}  .
\end{array}
\]
Note that when $r$ small enough, $B_{\xi,g}(x,r)$ is a trivial $S^{1}$-bundle
over the geodesic ball $B_{h}(\pi(x),r)$.

Based on Perelman's non-collapsing theorem for a transverse ball along the
unnormalizing Sasaki-Ricci flow, it follows that

\begin{lemma}
\label{L61} (\cite[Proposition 7.2]{co1}, \cite[Lemma 6.2]{he}, \cite[Lemma
3.14]{tz2}) Let $(M^{2n+1},\xi,g_{0})$ be a compact Sasakian manifold and let
$g^{T}(t)$ be the solution of the unnormalizing Sasaki-Ricci flow with the
initial transverse metric $g_{0}^{T}$. Then there exists a positive constant
$C$ such that for every $x\in M$, if $|R^{T}|\leq r^{-2}$ on $B_{\xi
,g(t)}(x,r)$ for $r\in(0,r_{0}]$, where $r_{0}$ is a fixed sufficiently small
positive number, then%
\[%
\begin{array}
[c]{c}%
\mathrm{Vol}(B_{\xi,g(t)}(x,r))\geq Cr^{2n}.
\end{array}
\]

\end{lemma}

Moreover, based on the $L^{2}$-bound of Riemannian curvature and
K\"{a}hler-Einstein condition on $Z_{\infty}$, we can say more about the limit
singular space $Z_{\infty}$ and then $M_{\infty}$. More precisely, once we
obtain (\ref{2022-e}), Theorem \ref{T51} and Lemma \ref{L61}, it follows from
arguments of Petersen-Wei \cite{pw1}, \cite{pw2}, Cheeger-Colding-Tian
\cite{cct} and \cite[Theorem 2.37]{tz1} that we have the following structure
theorem of limit spaces $Z_{\infty}$ and $M_{\infty}$.

\begin{theorem}
\label{T61} Let $(M_{i},\eta_{i},\xi,\Phi_{i},g_{i},\omega_{i})$ be a sequence
of quasi-regular Sasakian $(2n+1)$-manifolds with Sasaki metrics $g_{i}%
=g_{i}^{T}+\eta_{i}\otimes\eta_{i}$ such that for basic potentials
$\varphi_{i}$
\[%
\begin{array}
[c]{c}%
\eta_{i}=\eta+d_{B}^{C}\varphi_{i}%
\end{array}
\]
and
\[%
\begin{array}
[c]{c}%
d\eta_{i}=d\eta+\sqrt{-1}\partial_{B}\overline{\partial}_{B}\varphi_{i}.
\end{array}
\]
We denote that $(Z_{i},h_{i},J_{i},\omega_{h_{i}})$ are a sequence of
well-formed normal projective orbifolds of complex dimension $n$ which are the
corresponding foliation leave space with respect to $(M_{i},\eta_{i},\xi
,\Phi_{i},g_{i},\omega_{i})$ such that
\[%
\begin{array}
[c]{c}%
\frac{1}{2}d\eta_{i}=\pi^{\ast}(h_{i})=\pi^{\ast}(\omega_{h_{i}}),\text{ }%
\pi^{\ast}(J_{i})=\Phi_{i}.
\end{array}
\]
Suppose that $(M_{i},\eta_{i},\xi,\Phi_{i},g_{i},\omega_{i})$ is a smooth
transverse minimal model of general type satisfying
\begin{equation}%
\begin{array}
[c]{c}%
\int_{M}|\mathrm{Ric}_{g_{i}^{T}}^{T}+\omega_{i}|^{p}\omega_{i}^{n}\wedge
\eta\rightarrow0,
\end{array}
\label{a}%
\end{equation}
and
\begin{equation}%
\begin{array}
[c]{c}%
\mathrm{Vol}(\emph{B}_{\xi,g_{i}^{T}}(x_{i},1))\geq\nu
\end{array}
\label{b}%
\end{equation}
for some $p>n,$ $\upsilon>0$. Then passing to a subsequence if necessary,
$(M_{i},\Phi_{i},g_{i},x_{i})$ converges in the Cheeger-Gromov sense to limit
length spaces $(M_{\infty},\Phi_{\infty},d_{\infty},x_{\infty})$ and then
$(Z_{i},h_{i},J_{i},\pi(x_{i}))$ converges to $(Z_{\infty},h_{\infty
},J_{\infty},\pi(x_{\infty}))$ such that

(1) for any $r>0$ and $p_{i}\in M_{i}$ with $p_{i}\rightarrow p_{\infty}\in
M_{\infty},$%
\[%
\begin{array}
[c]{c}%
\mathrm{Vol}(B_{h_{i}}(\pi(p_{i}),r))\rightarrow\mathcal{H}^{2n}(B_{h_{\infty
}}(\pi(p_{\infty}),r))
\end{array}
\]
and
\[%
\begin{array}
[c]{c}%
\mathrm{Vol}(B_{\xi,g_{i}^{T}}(p_{i},r))\rightarrow\mathcal{H}^{2n}%
(B_{\xi,g_{\infty}^{T}}(p_{\infty},r)).
\end{array}
\]
Moreover,
\[%
\begin{array}
[c]{c}%
\mathrm{Vol}(B(p_{i},r))\rightarrow\mathcal{H}^{2n+1}(B(p_{\infty},r)),
\end{array}
\]
where $\mathcal{H}^{m}$ denotes the $m$-dimensional Hausdorff measure.

(2) $M_{\infty}$ is a $S^{1}$-orbibundle over the normal projective variety
$Z_{\infty}:=M_{\infty}/\mathcal{F}_{\xi}.$

(3) $Z_{\infty}=\mathcal{R}\cup\mathcal{S}$ such that $\mathcal{S}$ is a
closed singular set of codimension $4$ and $\mathcal{R}$ consists of points
whose tangent cones are $\mathbb{R}^{2n}.$

(4) the convergence on the regular part of $M_{\infty}$ which is a $S^{1}%
$-principle bundle over $\mathcal{R}$ in the $(C^{\alpha}\cap L_{B}^{2,p}%
)$-topology for any $0<\alpha<2-\frac{2n}{p}$.
\end{theorem}

\begin{proof}
Since $\xi$ is fixed and the metrics are under deformation generated by basic
potentials $\phi_{i}$ such that
\[%
\begin{array}
[c]{c}%
\eta_{i}=\eta+d_{B}^{C}\phi_{i}%
\end{array}
\]
and
\[%
\begin{array}
[c]{c}%
d\eta_{i}=d\eta+\sqrt{-1}\partial_{B}\overline{\partial}_{B}\phi_{i}.
\end{array}
\]
By t%
the first structure theorem on Sasakian manifolds, %
$M$ is a principal $S^{1}$-orbibundle ($V$-bundle) over $Z$ which is also a
$Q$-factorial, polarized, normal projective orbifold such that there is an
orbifold Riemannian submersion$\ \pi:(M,g_{i},\omega_{i})\rightarrow
(Z,h_{i},\omega_{h_{i}})$ with
\begin{equation}%
\begin{array}
[c]{c}%
g_{i}^{T}=\frac{1}{2}d\eta_{i}=\pi^{\ast}(h_{i})=\pi^{\ast}(\omega_{h_{i}})
\end{array}
\label{c}%
\end{equation}
and
\[%
\begin{array}
[c]{c}%
d_{i}^{T}(x,y)=d_{h_{i}}(\pi(x),\pi(y)).
\end{array}
\]
Then
\[%
\begin{array}
[c]{c}%
g_{i}=\pi^{\ast}h_{i}+\eta_{i}\otimes\eta_{i}.
\end{array}
\]
Hence by Cheeger-Colding-Tian structure theory for K\"{a}hler orbifolds
(\cite{cct}) that
\[%
\begin{array}
[c]{c}%
g_{i}^{T}=\pi^{\ast}h_{i}\rightarrow\pi^{\ast}h_{\infty}=g_{\infty}^{T}%
\end{array}
\]
and
\[%
\begin{array}
[c]{c}%
d_{i}^{T}=d_{g_{i}^{T}}^{T}\rightarrow d_{g_{\infty}^{T}}^{T}=d_{\infty}^{T}%
\end{array}
\]
as $i\rightarrow\infty.$ Thus
\[%
\begin{array}
[c]{c}%
\eta_{i}\rightarrow\eta_{\infty}%
\end{array}
\]
and
\[%
\begin{array}
[c]{c}%
g_{i}\rightarrow g_{\infty}=g_{\infty}^{T}+\eta_{\infty}\otimes\eta_{\infty}%
\end{array}
\]
as $i\rightarrow\infty.$ Moreover (\cite{cct}, \cite{co2}),
\[%
\begin{array}
[c]{c}%
g_{i}^{T}\overset{C^{\alpha}\cap L_{B}^{2,p}}{\rightarrow}g_{\infty}^{T}%
\end{array}
\]
such that
\[%
\begin{array}
[c]{c}%
h_{i}\overset{C^{\alpha}\cap L^{2,p}}{\rightarrow}h_{\infty}%
\end{array}
\]
with $g_{i}^{T}=\pi^{\ast}(h_{i}).$

Moreover,%
\[%
\begin{array}
[c]{c}%
B_{h_{i}}(\pi(x_{i}),r)\rightarrow B_{h_{\infty}}(\pi(x_{\infty}),r)
\end{array}
\]
and then
\[%
\begin{array}
[c]{c}%
B_{\xi,g_{i}^{T}}(x_{i},r)\rightarrow B_{\xi,g_{\infty}^{T}}(x_{\infty},r)
\end{array}
\]
as $i\rightarrow\infty.$ Furthermore,
\[%
\begin{array}
[c]{c}%
\mathrm{Vol}(B_{h_{i}}(\pi(x_{i}),r))\rightarrow\mathrm{Vol}(B_{h_{\infty}%
}(\pi(x_{\infty}),r))
\end{array}
\]
and then
\[%
\begin{array}
[c]{c}%
\mathrm{Vol}(B_{\xi,g_{i}^{T}}(x_{i},r))=\int_{B_{\xi,g_{i}^{T}}(x_{i}%
,r)}\omega_{i}^{n}\wedge\eta\rightarrow\int_{B_{\xi,g_{\infty}^{T}}(x_{\infty
},r)}\omega_{\infty}^{n}\wedge\eta=\mathcal{H}^{2n}(B_{\xi,g_{\infty}^{T}%
}(x_{\infty},r)).
\end{array}
\]
Finally
\[%
\begin{array}
[c]{c}%
\mathrm{Vol}(B(x_{i},r))=\int_{B(x_{i},r)}\omega_{i}^{n}\wedge\eta
\rightarrow\int_{B(x_{\infty},r)}\omega_{\infty}^{n}\wedge\eta=\mathcal{H}%
^{2n+1}(B(x_{\infty},r))
\end{array}
\]
as $i\rightarrow\infty.$

(3) and (4) will follow easily from (\ref{a}), (\ref{b}), (\ref{c}) and the
arguments as in \cite[Theorem 2.37]{tz1}.
\end{proof}

Let $(M,\eta,\xi,g)$ be a compact quasi-regular Sasakian $(2n+1)$-manifold and
be a principal $S^{1}$-orbibundle over $Z$ which is a well-formed
$Q$-factorial, polarized, normal projective orbifold.\ If $K_{M}^{T}$ is nef
and big (\cite{clw}), then it is semi-ample and then there exists a $S^{1}%
$-equivariant basic base point free holomorphic map
\[%
\begin{array}
[c]{c}%
\Psi:M\rightarrow(\mathbb{C}\mathrm{P}^{N},\omega_{FS})
\end{array}
\]
defined by the basic transverse holomorphic section $\{s_{0},s_{1}%
,\cdots,s_{N}\}$ of $H^{0}(M,(K_{M}^{T})^{m})$ with $N=\dim H^{0}(M,(K_{M}%
^{T})^{m})-1$ for a large positive integer $m$. Its image
\[%
\begin{array}
[c]{c}%
\Psi(M)=M_{\mathrm{can}}%
\end{array}
\]
is called the transverse canonical model of $M.$ Note that since $(M,\eta
,\xi,g)$ is a compact quasi-regular Sasakian manifold, $M$\ is a principal
$S^{1}$-orbibundle ($V$-bundle) over $Z$ which is also a $Q$-factorial,
polarized, normal projective orbifold such that there is an orbifold
Riemannian submersion$\ \pi:(M,g)\rightarrow(Z,h,\omega)$ with%
\[%
\begin{array}
[c]{c}%
c_{1}^{B}((K_{M}^{T})^{-1})=\pi^{\ast}c_{1}^{orb}(Z)=\pi^{\ast}c_{1}%
(K_{Z}^{-1}).
\end{array}
\]
Then there exists a base point free holomorphic map
\[%
\begin{array}
[c]{c}%
\widetilde{\Psi}:Z\rightarrow(\mathbb{C}\mathrm{P}^{N},\omega_{FS})
\end{array}
\]
defined by the holomorphic section $\{\widetilde{s}_{0},\widetilde{s}%
_{1},\cdots,\widetilde{s}_{N}\}$ of $H^{0}(Z,(K_{Z})^{m})$ and its image
\[%
\begin{array}
[c]{c}%
\widetilde{\Psi}(Z)=Z_{\mathrm{can}}%
\end{array}
\]
is the canonical model of $Z$ with the Kodaira dimension $\varkappa(Z)=n.$

Let $V\subset M$ be the basic exceptional locus of $\Psi$ with $\pi
(V)=E\subset Z$ which coincides with non-K\"{a}hler locus $\mathrm{Null}%
(-c_{1}^{B}(M))$ of the canonical class\ as in Proposition \ref{P31}
(\cite{ct}, \cite{tz3}). It follows from the first structure theorem
(Proposition \ref{P21}), Theorem \ref{T32}, Theorem \ref{T61}, and the
arguments as in \cite{tz2} that we have

\begin{theorem}
\label{T62}Let $(M,\eta,\xi,g)$ be a compact quasi-regular Sasakian $(2n+1)
$-manifold and be a principal $S^{1}$-orbibundle ($V$-bundle) over $Z$ which
is a well-formed $Q$-factorial, polarized, normal projective orbifold such
that there is an orbifold Riemannian submersion$\ \pi:(M,g,\omega
)\rightarrow(Z,h,\omega_{h}).$ Suppose that $(M,\eta,\xi,g)$ is a smooth
transverse minimal model of general type in the case of the dimension less
than or equal to $7$ and $\omega(t)$ be any solution to the Sasaki-Ricci flow
(\ref{3}). Then

\begin{enumerate}
\item $\omega(t)$ converges smoothly to a Sasaki $\eta$-Einstein metric
$\omega_{\infty}$ outside the exceptional locus $V;$

\item the metric completion of $(Z\backslash E,h_{\infty},\omega_{h_{\infty}%
})$ is homeomorphic to $Z_{\mathrm{can}}$ which is a normal projective variety
and then the metric completion of $(M\backslash V,g_{\infty},\omega_{\infty})$
is homeomorphic to $M_{\mathrm{can}}$, so it is compact.

\item $(Z_{\infty},d_{\infty}^{T})$ is isometric to the metric completion of
$(Z_{\infty}\backslash E,h_{\infty})$ and then $(M_{\infty},d_{\infty})$ is
isometric to the metric completion of $(M\backslash V,g_{\infty}%
,\omega_{\infty})$.
\end{enumerate}

As a consequence, for any sequence $t_{i}\rightarrow\infty,$ $(M,\omega
(t_{i}))$ converges along a subsequence in the Cheeger-Gromov sense to
\[%
\begin{array}
[c]{c}%
\mathrm{Ric}_{\omega_{\infty}}^{T}=-\omega_{\infty}%
\end{array}
\]
in the limit space $(M_{\infty},d_{\infty})$ which is the transverse canonical
model $M_{\mathrm{can}}$ of $M$.
\end{theorem}

Then our main results in this paper as in Theorem \ref{T11} and Corollary
\ref{C11} follows easily from Theorem \ref{T62}.

Finally, we add some remark about the Sasaki analogue of Guo-Song-Weinkove
\cite{gsw} arguments for the contraction on the foliation $(-2)$-curve for
$n=2$.\ Let $(M,\xi)$ be a compact quasi-regular Sasakian $5$-manifold.\ A
basic $1$-cycle $V$ on $M$\ is a formal finite sum $V=\sum a_{i}V_{i}$, for
$a_{i}\in\mathbb{Z}$ and $V_{i}$ is the irreducible invariant Sasakian
$3$-manifold. We denote by $N_{1}(M)_{\mathbb{Z}}$ the space of $1$-cycles
modulo numerical equivalence. Write%
\[%
\begin{array}
[c]{c}%
N_{1}(M)_{\mathbb{Q}}=N_{1}(M)_{\mathbb{Z}}\otimes_{\mathbb{Z}}\mathbb{Q}%
\text{\ \textrm{and}\ }N_{1}(M)_{\mathbb{R}}=N_{1}(M)_{\mathbb{Z}}%
\otimes_{\mathbb{Z}}\mathbb{R}\mathbf{.}%
\end{array}
\]
Then write $NE(M)$ for the cone of effective elements of $N_{1}(M)_{\mathbb{R}%
}$ and $\overline{NE(M)}$ for its closure. Furthermore, a basic divisor
$D^{T}$ is ample if and only if
\[%
\begin{array}
[c]{c}%
D^{T}\cdot V>0
\end{array}
\]
for all nonzero $V\in\overline{NE(M)}$. It is the Kleiman criterion for the
ample line bundle.

In view of the cohomological characterization of the maximal solution of the
Sasaki-Ricci flow (\ref{1}) (see section $3$), we start with a pair
$(M,H^{T}),$ where $M$ is a Sasakian manifold with an ample basic divisor
$H^{T}$. Let
\[%
\begin{array}
[c]{c}%
T_{0}=\sup\{t>0\mathbf{\ |\ }H^{T}+tK_{M}^{T}\text{ is \textrm{nef}}\}.
\end{array}
\]
Denote
\[%
\begin{array}
[c]{c}%
L_{0}^{T}:=H^{T}+T_{0}K_{M}^{T}%
\end{array}
\]
which is a basic $Q$-line bundle and semi-ample. In fact, it follows from
Kleiman criterion that$\ mL_{0}^{T}-T_{0}K_{M}^{T}$ is ample and then nef and
big for some sufficiently large $m$. Then by Kawamata criterion for base point
free, we have the semi-ample for $L_{0}^{T}.$

Next we define a subcone
\[%
\begin{array}
[c]{c}%
R:=\{V\in\overline{NE(M)}|\text{\ }L_{0}^{T}\cdot V=0\}
\end{array}
\]
which is a foliation extremal ray $R$ with the generic choice of $H^{T}.$
Moreover, we have $R=\overline{NE(M)}_{K_{M}^{T}<0}\cap(L_{0}^{T})^{\perp}. $
Then
\[%
\begin{array}
[c]{c}%
0=L_{0}^{T}\cdot V\Rightarrow K_{M}^{T}\cdot V=-\frac{1}{T_{0}}(H^{T}\cdot
V)<0.
\end{array}
\]
That is the map $\Psi$ induced from $(L_{0}^{T})^{m}$ contract all foliation
curves whose class lies in the foliation extremal ray $R$ with
\[%
\begin{array}
[c]{c}%
L_{0}^{T}\cdot V=0\text{ \textrm{and}\ }K_{M}^{T}\cdot V<0.
\end{array}
\]
The union of all foliation curves is called the locus of the foliation
extremal ray $R$ which is exactly the set of points where the map
$\Psi:M\rightarrow N$ is not isomorphism.

We observe that
\[%
\begin{array}
[c]{c}%
K_{M}^{T}\cdot V=0
\end{array}
\]
as $T_{0}\rightarrow\infty.$ Then the floating foliation $(-2)$-curves $V$
which is entirely contained in the smooth locus of $M$ with respective to the
foliation $\mathcal{F}_{\xi}$, .will be contracted to orbifold points by the
Sasaki-Ricci flow as $T_{0}\rightarrow\infty.$ We refer to \cite{clw} for some details.

\appendix

\section{ \ }

In this appendix, for a completeness, we will address the preliminary notions
on the Sasakian structure, the leave space and its foliation singularities,
basic holomorphic line bundles and basic divisors over Sasakian manifolds. We
refer to \cite{bg}, \cite{m}, and references therein for some details.

\subsection{Sasakian Structures, Leave Spaces and Its Foliation Singularities}

\begin{definition}
Let $(M,\eta,\xi,\Phi,g)$ be a compact Sasakian $(2n+1)$-manifold. If the
orbits of the Reeb vector field $\xi$ are all closed, and hence circles, then
integrates to an isometric $U(1)$ action on $(M,g)$. Since it is nowhere zero
this action is locally free; that is, the isotropy group of every point in $M$
is finite. If the $U(1)$ action is in fact free then the Sasakian structure is
said to be regular. Otherwise, it is said to be quasi-regular. It is said to
be irregular if the orbits of are not all closed. In this case the closure of
the $1$-parameter subgroup of the isometry group of $(M,g)$ is isomorphic to a
torus $T^{k}$, for some positive integer $k$ called the rank of the Sasakian
structure. In particular, irregular Sasakian manifolds have at least a $T^{2}$ isometry.
\end{definition}

Note that in the regular or quasi-regular case, the leaf space
$Z=M/\mathcal{F}_{\xi}=M/U(1)$ has the structure of a compact manifold or
orbifold, respectively. In the latter case the orbifold singularities of $Z$
descend from the points in $M$ with non-trivial isotropy subgroups which
finite subgroups of $U(1)$ and will all be\ isomorphic to cyclic groups. The
transverse K\"{a}hler structure described above then pushes down to a
K\"{a}hler structure on $Z$, so that $Z$ is a compact complex manifold or
orbifold equipped with a K\"{a}hler metric $h.$%

The first structure theorem on Sasakian manifolds states that %

\begin{proposition}
\label{P21}(\cite{ru}) Let $(M,\eta,\xi,\Phi,g)$ be a compact quasi-regular
Sasakian manifold of dimension $2n+1$ and $Z$ denote the space of leaves of
the characteristic foliation $\mathcal{F}_{\xi}$ (just as topological space). Then

(i) $Z$ carries the structure of a Hodge orbifold $\mathcal{Z=}(Z,\Delta)$
with an orbifold K\"{a}hler metric $h$ and K\"{a}hler form $\omega$ which
defines an integral class $[p^{\ast}\omega]$ in $H_{orb}^{2}(Z,\mathbb{Z}%
\mathbf{)}$ in such a way that $\pi:(M,g,\omega)\rightarrow(Z,h,\omega_{h})$
is an orbifold Riemannian submersion, and a principal $S^{1}$-orbibundle
($V$-bundle) over $Z.$ Furthermore, it satisfies $\frac{1}{2}d\eta=\omega
=\pi^{\ast}(\omega_{h}).$ The fibers of $\pi$ are geodesics.

(ii) $Z$ is also a $Q$-factorial, polarized, normal projective algebraic variety.

(iii) The orbifold $Z$ is Fano if and only if $\mathrm{Ric}_{g}>-2$. In this
case $Z$ as a topological space is simply connected and as an algebraic
variety is uniruled with Kodaira dimension $-\infty$.

(iv) $(M,\xi,g,\omega)$ is Sasaki-Einstein if and only if $(Z,h,\omega_{h})$
is K\"{a}hler-Einstein with scalar curvature $4n(n+1).$

(v) If $(M,\eta,\xi,\Phi,g)$ is regular then the orbifold structure is trivial
and $\pi$ is a principal circle bundle over a smooth projective algebraic variety.

(vi) As real cohomology classes, there is a relation between the basic Chern
class and orbifold Chern class
\[%
\begin{array}
[c]{c}%
c_{k}^{B}(M):=c_{k}(\emph{F}_{\xi})=\pi^{\ast}c_{k}^{orb}(Z).
\end{array}
\]

Conversely, for a compact Hodge orbifold $(Z,h)$. Let $\pi:M\rightarrow Z$ be
a principal $U(1)$-orbibundle over $Z$ whose first Chern class is an integral
class defined by $[\omega_{Z}]$, and let $\eta$ be a $1$-form on $M$ with
$\frac{1}{2}d\eta=\pi^{\ast}\omega_{h}$ (is then $\eta$ proportional to a
connection $1$-form). Then $(M,\pi^{\ast}h+\eta\otimes\eta)$ is a Sasakian
orbifold. Furthermore, if all the local uniformizing groups inject into the
structure group $U(1)$, then the total space $M$ is a smooth manifold.
\end{proposition}

Note that in the quasi-regular case, the projection is instead a principal
$U(1)$ orbibundle, with $\omega_{Z}$ again proportional to a curvature
$2$-form. The orbifold cohomology group $H_{orb}^{2}(Z,\mathbb{Z})$ classifies
isomorphism classes of principal $U(1)$ orbibundles over an orbifold $Z$, just
as in the regular manifold case the first Chern class in $H^{2}(Z,\mathbb{Z})$
classifies principal $U(1)$ bundles. The K\"{a}hler form $H_{orb}%
^{2}(Z,\mathbb{Z})$ then defines a cohomology class $[\omega_{Z}]$ in
$H^{2}(Z,\mathbb{R})$ which is proportional to a class in the image of the
natural map
\[%
\begin{array}
[c]{c}%
p:H_{orb}^{2}(Z,\mathbb{Z})\rightarrow H_{orb}^{2}(Z,\mathbb{R})\rightarrow
H^{2}(Z,\mathbb{R}).
\end{array}
\]

On the other hand, the second
structure theorem on Sasakian manifolds states that %

\begin{proposition}
(\cite{ru}) Let $(M,g)$ be a compact Sasakian manifold of dimension $2n+1$.
Any Sasakian structure $(\xi,\eta,\Phi,g)$ on $M$ is either quasi-regular or
there is a sequence of quasi-regular Sasakian structures $(\xi_{i},\eta
_{i},\Phi_{i},g_{i})$ converging in the compact-open $C^{\infty}$-topology to
$(\xi,\eta,\Phi,g).$ In particular, if $M$ admits an irregular Sasakian
structure, it admits many locally free circle actions.
\end{proposition}

We recall that

\begin{definition}
\label{D2021}An orbifold complex manifold is a normal, compact, complex space
$Z$ locally given by charts written as quotients of smooth coordinate charts.
That is, $Z$ can be covered by open charts $Z=\cup U_{i}.$ The orbifold charts
on $(Z,U_{i},\varphi_{i})$ is defined by the local uniformizing systems
$(\widetilde{U_{i}},G_{i},\varphi_{i})$ centered at the point $p_{i}$, where
$G_{_{i}}$ is the local uniformizing finite group acting on a smooth complex
space $\widetilde{U_{i}}$ such that $\varphi_{i}:\widetilde{U_{i}}\rightarrow
U_{i}=\widetilde{U_{i}}/G_{_{i}}$ is the biholomorphic map. A point $x$ of
complex orbifold $X$ whose isotropy subgroup $\Gamma_{x}\neq Id$ is called a
singular point. Those points with $\Gamma_{x}=Id$ are called regular points.
The set of singular points is called the orbifold singular locus or orbifold
singular set, and is denoted by $\Sigma^{orb}(Z)$.
\end{definition}

Let $\Gamma\subset GL(n,\mathbb{C})$ be a finite subgroup. Then the quotient
space $\mathbb{C}^{n}/\Gamma$ is smooth if and only if $\Gamma$ is a
reflection group which fixes a hyperplane in $\mathbb{C}^{n}$. Now let
$(M,\eta,\xi,\Phi,g)$ be a compact quasi-regular Sasakian manifold of
dimension $2n+1$. By the first structure theorem, the underlying complex space
$\mathcal{Z}=(Z,U_{i})$ is a normal, orbifold variety with the
algebro-geometric singular set $\Sigma(Z)$. Then $\Sigma(Z)\subset\Sigma
^{orb}(Z)$ and it follows that $\Sigma(Z)=\Sigma^{orb}(Z)$ if and only if none
of the local uniformizing groups $\Gamma_{i}$ of the orbifold $\mathcal{Z}%
=(Z,U_{i})$ contain a reflection. If some $\Gamma_{i}$ contains a reflection,
then the reflection fixes a hyperplane in $\widetilde{U_{i}}$ giving rise to a
ramification divisor on $\widetilde{U_{i}}$ and a branch divisor on $Z.$

\begin{definition}
(i) The branch divisor $\Delta$ of an orbifold $\mathcal{Z}=(Z,\Delta)$ is a
$Q$-divisor on $Z$ of the form
\[%
\begin{array}
[c]{c}%
\Delta=\sum_{\alpha}(1-\frac{1}{m_{\alpha}})D_{\alpha},
\end{array}
\]
where the sum is taken over all Weil divisors $D_{\alpha}$ that lie in the
orbifold singular locus $\Sigma^{orb}(Z)$, and $m_{\alpha}$ is the $gcd$ of
the orders of the local uniformizing groups taken over all points of
$D_{\alpha}$ and is called the ramification index of $D_{\alpha}.$

(ii) The orbifold structure $\mathcal{Z}=(Z,\Delta)$ is called well-formed if
the fixed point set of every non-trivial isotropy subgroup has codimension at
least two. That is, $\mathcal{Z}=(Z,\emptyset).$ Then $Z$ is well-formed if
and only if its orbifold singular locus and algebro-geometric singular locus
coincide, equivalently $Z$ has no branch divisors.
\end{definition}

\begin{example}
For instance, the weighted projective $CP(1;4;6)$ has a branch divisor
$\frac{1}{2}D_{0}=\{z_{0}=0\}.$ But $CP(1;2;3)$ is a unramified well-formed
\ orbifold with two singular points, $(0;1;0)$ with local uniformizing group
the cyclic group $\mathbb{Z}_{2}$, and $(0;0;1)$ with local uniformizing group
$\mathbb{Z}_{3}$.
\end{example}

Note that the orbifold canonical divisor $K_{\emph{Z}}^{orb}$ and canonical
divisor $K_{Z}$ are related by
\begin{equation}
K_{\emph{Z}}^{orb}=\varphi^{\ast}(K_{Z}+[\Delta]).\label{orbifold}%
\end{equation}
In particular, $K_{\emph{Z}}^{orb}=\varphi^{\ast}K_{Z}$ if and only if there
are no branch divisors.

For all previous discussions with the special case for $n=2,$ we have the
following result concerning its foliation cyclic quotient singularities.

\begin{theorem}
\label{T21}(\cite{clw}) Let $(M,\eta,\xi,\Phi,g)$ be a compact quasi-regular
Sasakian $5$-manifold and $Z$ be its leave space of the characteristic
foliation\textbf{. } Then $Z$ is a $Q$-factorial normal projective algebraic
orbifold surface satisfying

\begin{enumerate}
\item if its leave space $(Z,\emptyset)$ has at least codimension two fixed
point set of every non-trivial isotropy subgroup. That is to say $Z$ is
well-formed, then $Z$ has isolated singularities of a finite cyclic quotient
of $\mathbb{C}^{2}$ and the action is
\[
\mu_{Z_{r}}:(z_{1},z_{2})\rightarrow(\zeta^{a}z_{1},\zeta^{b}z_{2}),
\]
where $\zeta$ is a primitive $r$-th root of unity. We denote the cyclic
quotient singularity by $\frac{1}{r}(a,b)$ with $(a,r)=1=(b,r)$. In
particular, the action can be rescaled so that every cyclic quotient
singularity corresponds to a $\frac{1}{r}(1,a)$-point with $(r,a)=1,$
$\zeta=e^{\frac{2\pi i}{r}}$. In particular, it is klt (Kawamata log terminal)
singularities. More precisely, the corresponding singularities in $(M,\eta
,\xi,\Phi,g)$\ is called\ foliation cyclic quotient singularities of type
\[%
\begin{array}
[c]{c}%
\frac{1}{r}(1,a)
\end{array}
\]
at a singular fibre $S_{p}^{1}$ in $M$.

\item if its leave space $(Z,\Delta)$ has the codimension one fixed point set
of some non-trivial isotropy subgroup. Then the action is
\[%
\begin{array}
[c]{c}%
\mu_{Z_{r}}:(z_{1},z_{2})\rightarrow\left(  e^{\frac{2\pi\sqrt{-1}a_{1}}%
{r_{1}}}z_{1},e^{\frac{2\pi\sqrt{-1}a_{2}}{r_{2}}}z_{2}\right)  ,
\end{array}
\]
for some positive integers $r_{1},$ $r_{2}$ whose least common multiplier is
$r$, and $a_{i},$ $i=1,2$ are integers coprime to $r_{i},$ $i=1,2$. Then the
foliation singular set contains some $3$-dimensional submanifolds of $M.$ More
precisely, the corresponding singularities in $(M,\eta,\xi,\Phi,g)$%
\textbf{\ }is called\ the Hopf $S^{1}$-orbibundle over a Riemann surface
$\Sigma_{h}.$
\end{enumerate}
\end{theorem}

\begin{definition}
\label{D21}(i) Let $(M,\eta,\xi,\Phi,g)$ be a compact quasi-regular Sasakian
$5$-manifold and its leave space $(Z,\emptyset)$ of the characteristic
foliation be well-formed. Then\ the corresponding singularities in
$(M,\eta,\xi,\Phi,g)$\ is called\ foliation cyclic quotient singularities of
type $\frac{1}{r}(1,a)$ at a singular fibre $S_{p}^{1}$ in $M$. The foliation
singular set is discrete, and hence finite. It is klt (Kawamata log terminal) singularities.

(ii) Let $(M,\eta,\xi,\Phi,g)$ be a compact quasi-regular Sasakian
$5$-manifold and its leave space $(Z,\Delta)$ has the codimension one fixed
point set of some non-trivial isotropy subgroup. Then the foliation singular
set contains some $3$-dimensional submanifolds of $M.$ More precisely, the
corresponding singularities in $(M,\eta,\xi,\Phi,g)$\textbf{\ }is called\ the
Hopf $S^{1}$-orbibundle over a Riemann surface $\Sigma_{g}.$
\end{definition}

It follows\ from the Hirzebruch Jung continued fraction (\cite{r}) that

\begin{theorem}
\label{T22}(\cite{clw}) Let $(M,\eta,\xi,\Phi,g)$ be a compact quasi-regular
Sasakian $5$-manifold $M$ with a foliation cyclic quotient singularity of type
$\frac{1}{r}(1,a)$ at a singular fibre $S_{p}^{1},$ then the Hirzebruch Jung
continued fraction
\[%
\begin{array}
[c]{c}%
\frac{r}{a}=[b_{1},\cdots,b_{l}]
\end{array}
\]
gives the information on the foliation minimal resolution $\varphi
:\widetilde{M}\rightarrow M$ of $M$. We denote the exceptional foliation
curves $V_{i}$ of such a resolution. Then the exceptional foliation curves
form a chain of $\{V_{1},\cdots,V_{l}\}$ such that each $V_{i}$ has self
intersection $V_{i}^{2}=-b_{i}$ for every $i=1,\cdots,l$ and $V_{i}$
intersects another foliation curve $V_{j}$ transversely only if $j=i-1$ or
$j=i+1.$ In particular, for a foliation cyclic quotient singularity of type
$\frac{1}{k}(1,1),$ we have $\frac{k}{1}=[k]$ as the foliation $(-k)$-curve in
$\widetilde{M}$.
\end{theorem}

\subsection{Basic Holomorphic Line Bundles, Basic Divisors on Sasakian
Manifolds}

Let $(M,\eta,\xi,\Phi,g)$ be a compact Sasakian $(2n+1)$-manifold. We define
$D=\ker\eta$ to be the holomorphic contact vector bundle of $TM$ such that
\[
TM=D\oplus\left\langle \xi\right\rangle =T^{1,0}(M)\oplus T^{0,1}%
(M)\oplus\left\langle \xi\right\rangle .
\]
Then its associated strictly pseudoconvex CR $(2n+1)$-manifold to be denoted
by $(M,T^{1,0}(M),\xi,\Phi).$

\begin{definition}
(\cite{ta}) Let $(M,T^{1,0}(M))$ be a strictly pseudoconvex CR $(2n+1)$%
-manifold and $E\rightarrow M$ be a $C^{\infty}$ complex vector bundle over
$M.$ A pair $(E,\overline{\partial}_{b})$ is a CR-holomorphic vector bundle if
the differential operator
\[
\overline{\partial}_{b}:\Gamma^{\infty}(E)\rightarrow\Gamma^{\infty}%
(T^{0,1}(M)^{\ast}\otimes E)
\]
is defined by%
 %

(i)
\[
\overline{\partial}_{\overline{Z}}(fs)=(\overline{\partial}_{b}f)(\overline
{Z})\otimes s+f\overline{\partial}_{\overline{Z}}s,
\]
(ii)%
\[
\overline{\partial}_{\overline{Z}}\overline{\partial}_{\overline{W}%
}s-\overline{\partial}_{\overline{W}}\overline{\partial}_{\overline{Z}%
}s-\overline{\partial}_{[\overline{Z},\overline{W}]}s=0,
\]
for any $f\in C^{\infty}(M)\otimes\mathbb{C}$, $s\in\Gamma^{\infty}(E)$ and
$Z,$ $W\in\Gamma^{\infty}(T^{1,0}(M))$.
\end{definition}

The condition (ii) of the definition means that $(0,2)$-component of the
curvature operator $R(E)$ is vanishing when $E$ admits a connection $D$ whose
$(0,1)$-part is the operator $\overline{\partial}_{b}$ as in the following Lemma.

\begin{lemma}
Let $(M,T^{1,0}(M),\theta)$ be a strictly pseudoconvex CR $(2n+1)$-manifold
and $(E,\overline{\partial}_{b})$ a CR-holomorphic vector bundle over $M$. Let
$h=\left\langle \cdot,\cdot\right\rangle $ be a Hermitian structure in $E$.
Then there exists a unique (Tanaka) connection $D$ in $E$ such that

(i)
\[
D_{\overline{Z}}s=(\overline{\partial}_{b}s)\overline{Z},
\]
(ii)
\[
Z\left\langle s_{1},s_{2}\right\rangle =\left\langle D_{Z}s_{1},s_{2}%
\right\rangle +\left\langle s_{1},D_{\overline{Z}}s_{2}\right\rangle ,
\]
(iii) The $(0,2)$-component of the curvature operator $\Theta(E)$ is
vanishing. Here $\Theta(E):=D^{2}s.$
\end{lemma}

\begin{definition}
Let $(M,\eta,\xi,\Phi,g)$ be a compact Sasakian $(2n+1)$-manifold. A
CR-holomorphic vector bundle $(E,\overline{\partial}_{b})$ over $M$ is a basic
transverse holomorphic vector bundle over $(M,T^{1,0}(M))$ if\ there exists an
open cover $\{U_{\alpha}\}$ of $M$ and the trivializing frames on $U_{\alpha}%
$, such that its transition functions are matrix-valued basic CR functions.
The trivializing frames is called the basic transverse holomorphic frame.
\end{definition}

\begin{example}
Let $(M,\eta,\xi,\Phi,g)$ be a compact Sasakian $(2n+1)$-manifold. Then, with
respect to the trivializing frames
\[%
\begin{array}
[c]{c}%
\{\frac{\partial}{\partial x},\text{ }Z_{j}=\left(  \frac{\partial}{\partial
z^{j}}+\sqrt{-1}h_{j}\frac{\partial}{\partial x}\right)  ,\ j=1,2,\cdots,n\}
\end{array}
\]
the transition functions of such frames are basic transverse holomorphic
functions, that is $h$ is basic. Thus $T^{1,0}(M)$ is a basic transverse
holomorphic vector bundle. Moreover, the canonical (determinant) bundle
$K_{M}^{T}$ of $T^{1,0}(M)$ is a basic transverse holomorphic line bundle
whose transition functions are given by $t_{\alpha\beta}=\det(\partial
z_{\beta}^{i}/\partial z_{\alpha}^{j})$ on $U_{\alpha}\cap U_{\beta}$, where
$(x,z_{\alpha}^{1},\cdots,z_{\alpha}^{n})$ is the normal coordinate on
$U_{\alpha}.$
\end{example}

\begin{definition}
(i) Let $(M,\eta,\xi,\Phi,g)$ be a Sasakian $(2n+1)$-manifold and $L$ be a
basic transverse holomorphic bundle over $M$. A basic transverse holomorphic
section $s$ of $L$ is a collection $\{s_{\alpha}\}$ of CR-holomorphic maps
$s_{\alpha}:U_{\alpha}\rightarrow\mathbb{C}$ satisfying the transformation
rule $s_{\alpha}=t_{\alpha\beta}s_{\beta}$ on $U_{\alpha}\cap U_{\beta}$. The
transition function $t_{\alpha\beta}$ is basic. A basic Hermitian metric $h$
on $L$ is a collection $\{h_{\alpha}\}$ of smooth positive functions
$h_{\alpha}:U_{\alpha}\rightarrow\mathbb{R}$ satisfying the transformation
rule
\[
h_{\alpha}=|t_{\beta\alpha}|^{2}h_{\beta}%
\]
on $U_{\alpha}\cap U_{\beta}$. Given a basic transverse holomorphic section
$s$ and a Hermitian metric $h$, we can define the pointwise norm squared of
$s$ with respect to $h$ by
\[
|s|_{h}^{2}=h_{\alpha}s_{\alpha}\overline{s_{\alpha}}%
\]
on $U_{\alpha}$. The reader can check that $|s|_{h}^{2}$ is a well-defined
function on $M$.

(ii) A Hermitian metric is called a basic Hermitian metric if $h_{\alpha}$ is
basic. It always exists if $L$ is a basic transverse holomorphic line bundle.

(iii) We define the curvature $R_{h}^{T}$ of a basic Hermitian metric $h$ on
$L$ to be the basic closed $(1,1)$-form on $M$ given by
\[%
\begin{array}
[c]{c}%
R_{h}^{T}=-\frac{\sqrt{-1}}{2\pi}\partial_{B}\overline{\partial}_{B}\log
h_{\alpha}%
\end{array}
\]
on $U_{\alpha}$. This is well-defined. The basic first Chern class $c_{1}%
^{B}(L)$ of $L$ to be the cohomology class $[R_{h}^{T}]_{B}\in H_{\overline
{\partial}_{B}}^{1,1}(M,\mathbb{R})$. Since any two basic Hermitian metrics
$h,$ $h^{\prime}$ on $L$ are related by $h^{\prime}=e^{-\phi}h$ for some
smooth basic function $\phi$, we see that $R_{h^{\prime}}^{T}=R_{h}^{T}%
+\frac{\sqrt{-1}}{2\pi}\partial_{B}\overline{\partial}_{B}\phi$ and hence
$c_{1}^{B}(L)$ is well-defined, independent of choice of basic Hermitian
metric $h.$ We say that $(L,h)$ is positive if the curvature $R_{h}^{T}$ is
positive definite at every $p\in M.$
\end{definition}

\begin{example}
Let $(M,\eta,\xi,\Phi,g)$ be a compact Sasakian $(2n+1)$-manifold. If $g^{T}$
is a transverse K\"{a}hler metric on $M,$ then $h_{\alpha}=\det((g_{i\overline
{j}}^{\alpha})^{T})$ on $U_{\alpha}$ defines a basic Hermitian metric on the
canonical bundle $K_{M}^{T}$. The inverse $(K_{M}^{T})^{-1}$ of $K_{M}^{T}$ is
sometimes called the anti-canonical bundle. Its basic first Chern class
$c_{1}^{B}((K_{M}^{T})^{-1})$ is called the basic first Chern class of $M$ and
is often denoted by $c_{1}^{B}(M).$ Then it follows from the previous result
that $c_{1}^{B}(M)=[\mathrm{Ric}^{T}(\omega)]_{B}$ for any transverse
K\"{a}hler metric $\omega$ on a Sasakian manifold $M$.
\end{example}

\begin{definition}
(i) Let $(L,h)$ be a basic transverse holomorphic line bundle over a Sasakian
manifold $(M,\eta,\xi,\Phi,g)$ with the basic Hermitian metric $h.$ We say
that $L$ is very ample if for any ordered basis $\underline{s}=(s_{0}%
,\cdots,s_{N})$ of $H_{B}^{0}(M,L)$, the map $i_{\underline{s}}:M\rightarrow
\mathbb{C}\mathrm{P}^{N}$ given by%
\[
i_{\underline{s}}(x)=[s_{0}(x),\cdots,s_{N}(x)]
\]
is well-defined and an embedding which is $S^{1}$-equivariant with respect to
the weighted $\mathbb{C}^{\ast}$action in $\mathbb{C}^{N+1}$ as long as not
all the $s_{i}(x)$ vanish. We say that $L$ is ample if there exists a positive
integer $m_{0}$ such that $L^{m}$ is very ample for all $m\geq m_{0}.$

(ii) $L$ is a semi-ample basic transverse holomorphic line bundle if there
exists a basic Hermitian metric $h$ on $L$ such that $R_{h}^{T}$ is a
nonnegative $(1,1)$-form. In fact, there exists a $S^{1}$-equivariant
foliation base point free holomorphic map
\[
\Psi:M\rightarrow(\mathbb{C}\mathrm{P}^{N},\omega_{FS})
\]
defined by the basic transverse holomorphic section $\{s_{0},s_{1}%
,\cdots,s_{N}\}$ of $H_{B}^{0}(M,L^{m})$ which is $S^{1}$-equivariant with
respect to the weighted $\mathbb{C}^{\ast}$action. Here $N=\dim H_{B}%
^{0}(M,L^{m})-1 $ for a large positive integer $m$ and
\[%
\begin{array}
[c]{c}%
0\leq\frac{1}{m}\Psi^{\ast}(\omega_{FS})=\widehat{\omega}_{\infty}\in
c_{1}^{B}(L).
\end{array}
\]

\end{definition}

There is a Sasakian analogue of
Kodaira embedding theorem on %
a compact quasi-regular Sasakian $(2n+1)$-manifold
due to %
\cite{rt}, \cite{hlm}:

\begin{proposition}
Let $(M,\eta,\xi,\Phi,g)$ be a compact quasi-regular Sasakian $(2n+1)$%
-manifold and $(L,h)$ be a basic transverse holomorphic line bundle over $M$
with the basic Hermitian metric $h.$ Then $L$ is ample if and only if $L$ is positive.
\end{proposition}

\begin{definition}
(i) First we say that a subset $V$ of a (quasi-regular) Sasakian
$(2n+1)$-manifold $(M,\eta,\xi,\Phi,g)$ an invariant (Sasakian) submanifold
(with or without singularities) of dimension $2n-1$ if $\xi$ is tangent to $V$
and $\Phi TV\subset TV$ at all points of $V$ and is locally given as the zero
set $\{f=0\}$ of a locally defined basic CR holomorphic function $f$.\ In
general, $V$ may not be a submanifold. Denote by $V^{reg}$ the set of points
$p\in V$ for which $V$ is a submanifold of $M$ near $p$. We say that $V$ is
irreducible if $V^{reg}$ is connected. A transverse divisor $D^{T}$ on $M$ is
a formal finite sum $\sum_{i}a_{i}V_{i}$ where $a_{i}\in\mathbb{Z}$ and each
$V_{i}$ is an irreducible invariant submanifold of dimension $2n-1$. We say
that $D^{T}$ is effective if the $a_{i}$ are all nonnegative. The support of
$D^{T}$ is the union of the $V_{i}$ for each $i$ with $a_{i}\neq0.$

(ii) Given a transverse divisor $D^{T},$ we define an associated line bundle
as follows. Suppose that $D^{T}$ is given by local defining basic functions
$f_{\alpha}$ (vanishing on $D^{T}$ to order $1$) over an open cover
$U_{\alpha}$. Define transition functions $f_{\alpha}=t_{\alpha\beta}f_{\beta
}$ on $U_{\alpha}\cap U_{\beta}$. These are basic CR holomorphic and
nonvanishing in $U_{\alpha}\cap U_{\beta}$, and satisfy
\[
t_{\alpha\beta}t_{\beta\alpha}=1;\text{ }t_{\alpha\beta}t_{\beta\gamma
}=t_{\alpha\gamma}.
\]
Write $[D^{T}]$ for the associated basic line bundle, which is well-defined
independent of choice of local defining functions.

(iii) One can define
\[%
\begin{array}
[c]{c}%
L_{M}\cdot V=\int_{V}R_{h}^{T}\wedge\eta
\end{array}
\]
for all invariant Sasakian $3$-manifold $V$ in $M.$ Here $h$ is a basic
Hermitian metric on the basic line bundle $L_{M}.$\ From (ii), for a compact
Sasakian $5$-manifold $M,$ a transverse divisor $D^{T}$ defines an element of
$H_{\overline{\partial}_{B}}^{1,1}(M,\mathbb{R})$ by $D^{T}\rightarrow\lbrack
R_{h}^{T}]\in H_{\overline{\partial}_{B}}^{1,1}(M,\mathbb{R})$ for a basic
Hermitian metric on the associate basic line bundle $[D^{T}]$, and we define
\[%
\begin{array}
[c]{c}%
\alpha\cdot\beta=\int_{M}\alpha\wedge\beta\wedge\eta
\end{array}
\]
for $\alpha,$ $\beta\in H_{\overline{\partial}_{B}}^{1,1}(M,\mathbb{R}).$ Then
for an invariant $3$-manifold $V$ which is both a foliation curve and a
transverse divisor, then $V\cdot V$ is well-defined and we may write $V^{2}$
instead of $V\cdot V.$
\end{definition}

\begin{remark}
(\cite{gei}) Note that the Sasakian $3$-manifold $V$ is either canonical,
anticanonical or null. $V$ is up to finite quotient a regular Sasakian
$3$-manifold, i.e., a circle bundle over a Riemann surface of positive genus.
In the positive case, $V$ is covered by $S^{3}$ and its Sasakian structure is
a deformation of a standard Sasakian structure.
\end{remark}

\begin{definition}
(i) We say that a basic line bundle $L$ is nef if $L\cdot V\geq0$ for any
invariant Sasakian $3$-manifold $V$ in $M.$ In particular if $M$ is
quasi-regular, then $V$ is the $S^{1}$-oribundle over the curve $C$ in $Z$ so
that
\[%
\begin{array}
[c]{c}%
L_{Z}\cdot C=\int_{C}R_{h_{Z}}\geq0.
\end{array}
\]
Here $c_{1}^{B}(L_{M})=\pi^{\ast}c_{1}^{orb}(L_{Z})$ and $h_{Z}$ is the
Hermitian metric in the corresponding line bundle $L_{Z}.$ Define
\begin{equation}
C_{M}^{B}=\{[\alpha]_{B}\in H_{B}^{1,1}(M,\mathbb{R})|\text{ }\exists\text{
}\omega>0\text{ such that }[\omega]_{B}=[\alpha]_{B}\}.\label{D}%
\end{equation}
Then we can also define a class $[\alpha]_{B}$ called nef class if
$[\alpha]_{B}\in\overline{C_{M}^{B}}$ and a class $[\alpha]_{B}$ called big
if
\[%
\begin{array}
[c]{c}%
\int_{M}\alpha^{n}\wedge\eta>0.
\end{array}
\]

(ii) If the Sasakian $(2n+1)$-manifold $(M,\eta,\xi,\Phi,g)$ has the canonical
basic line bundle $K_{M}^{T}$ nef, then we say that $M$ is a smooth transverse
minimal model. If $M$ has $K_{M}^{T}$ big, then we say that $M$ is of general type.
\end{definition}

\end{document}